\documentclass[12pt]{amsart}

\usepackage{amsmath, amsfonts, amssymb, amsthm}
\usepackage{amstext} 
\usepackage{mathrsfs}
\usepackage{mathtools}
\usepackage{xcolor}
\usepackage{ulem}
\usepackage{natbib}
\usepackage{graphicx}
\usepackage{subfig}
\usepackage{url}

\usepackage[hmargin=1in, vmargin=1in]{geometry}

\newtheorem{theorem}{Theorem}[section]
\newtheorem{Lemma}{Lemma}
\newtheorem{Corollary}{Corollary}
\newtheorem{assumption}{Assumption}

\makeatletter
\DeclareRobustCommand*\cal{\@fontswitch\relax\mathcal}
\makeatother

\DeclareMathOperator\argmax{argmax}
\DeclareMathOperator\argmin{argmin}
\DeclareMathOperator\var{Var}
\def\bibfont{\small}

\graphicspath{{./art/}}

\usepackage{siunitx}

\title{On identifiability and consistency of the nugget in Gaussian spatial process models}

\author{Wenpin Tang}
\author{Lu Zhang}\thanks{The first and second authors have equal contributions to this paper.}
\author{Sudipto Banerjee}

\begin{document}

\maketitle

\begin{abstract}
Spatial process models popular in geostatistics often represent the observed data as the sum of a smooth underlying process and white noise. The variation in the white noise is attributed to measurement error, or micro-scale variability, and is called the ``nugget''. We formally establish results on the identifiability and consistency of the nugget in spatial models based upon the Gaussian process within the framework of in-fill asymptotics, i.e. the sample size increases within a sampling domain that is bounded. Our work extends results in fixed domain asymptotics for spatial models without the nugget. More specifically, we establish the identifiability of parameters in the Mat\'ern covariogram and the consistency of their maximum likelihood estimators in the presence of discontinuities due to the nugget. We also present simulation studies to demonstrate the role of the identifiable quantities in spatial interpolation.
\end{abstract}


\section{Introduction}
The analysis of point-referenced spatial data relies heavily on stationary Gaussian processes for modelling spatial dependence. Let $y(s)$ be the outcome measured at a location $s \in {\cal S} \subset \mathbb{R}^d$, where ${\cal S}$ is a bounded region within $\mathbb{R}^d$. The outcome is customarily modelled as   
\begin{equation}\label{eq:basic}
y(s) = \mu(s) + w(s) + \epsilon(s), \quad s \in S \subset \mathbb{R}^d\;,
\end{equation}
where $\mu(s)$ models the trend, $w(s)$ is a Gaussian process capturing spatial dependence, and $\epsilon(s)$ is a white noise process modelling measurement error or micro-scale variation. \cite{Matern86} introduced a flexible class of covariance functions for modelling $w(s)$ that has been widely used in spatial modelling ever since it was recommended in \cite{Stein99}. The finite dimensional realizations of $\epsilon(s)$ are modelled independently and identically as $N(0,\tau^2)$ over any finite collection of locations. The variance parameter $\tau^2$ is called the ``nugget''. 

Our intended contribution in this article is to formally establish the identifiability and consistency of the process parameters in (\ref{eq:basic}) in the presence of an unknown nugget under infill or fixed domain asymptotics, where the sample size increases with increasing numbers of locations within a domain that is fixed and does not expand.  
This distinguishes the article from existing results on inference for process parameters in Mat\'ern models that have, almost exclusively, been studied without the presence of an unknown nugget. 
\cite{zhang2005towards} compared infill and expanding domain asymptotic paradigms and elucidate a preference for the former for analysing the limiting distributions of parameters in the Mat\'ern family. \cite{Zhang04} showed that not all parameters in the Mat\'ern family can be consistently estimated under infill asymptotics, 
but certain microergodic parameters, which play a crucial role in the identifiability of Gaussian processes with the Mat\'ern covariogram (see Section~\ref{subsec: identifiability} for further details), are consistently estimable.
\cite{DZM09} derived the asymptotic normality of the maximum likelihood estimator for such microergodic parameters. 
\cite{KS13} extended these asymptotic results to the case of jointly estimating the spatial range and the variance parameters in the Mat\'ern family, and explored the effect of a prefixed range verses a joint estimated range on inference when having relatively small sample size.
{
Recently \cite{BFFP19} and \cite{MB19} considered more general classes of covariance functions outside of the Mat\'ern family  and studied the consistency and asymptotic normality of the maximum likelihood estimator for the corresponding microergodic parameters.}

These studies have focused upon settings without the presence of a nugget. 
In practice, modelling the measurement error, or nugget effect, in \eqref{eq:basic} is prevalent in geostatistical modelling. 
{
The main difference between the model without a nugget and that with a nugget hinges on the rate of asymptotic normality of the maximum likelihood estimator of microergodic parameters: the former has a universal rate of $n^{1/2}$, while the latter, as shown in Theorem \ref{thm:CLT}, has a rate of $n^{1/(2+ 4 \nu/d)}$ which depends on the model parameters. We also note that deriving the rate of $n^{1/(2+ 4 \nu/d)}$ for a Mat\'ern model with a nugget effect is not an obvious consequence of any aforementioned results for Mat\'ern or Mat\'ern-like models without a nugget effect. Previous to this work,} \cite{zhang2005towards} offered some heuristic arguments for the consistency and asymptotic normality of the maximum likelihood estimators of microergodic parameters in \eqref{eq:basic}. \cite{CSY00} demonstrated that the presence of measurement error can have a big impact on the parameter estimates for Ornstein-Uhlenbeck processes, {
i.e., Mat\'ern processes with $\nu = 1/2$ and $d = 1$}, over bounded intervals. 
Their proof exploits the Markovian property and the explicit formula for the maximum likelihood estimator of the one-dimensional Ornstein-Uhlenbeck process that are not available in the case of the Mat\'ern model over $\mathbb{R}^d$ with $d \geq 2$. 


Returning to (\ref{eq:basic}), it will be sufficient for our subsequent development to assume that $\mu(s)=0$, i.e., the data have been de-trended. We specify $\{w(s) : s \in {\cal S} \subset \mathbb{R}^d \}$ as a zero-centered stationary Gaussian process with isotropic Mat\'ern covariogram,
\begin{equation}
\label{eq:Matern}
K_w(x; \, \sigma^2, \phi, \nu) := \frac{\sigma^2 (\phi \|x\|)^{\nu}}{\Gamma(\nu) 2^{\nu - 1}} K_{\nu}(\phi \|x\|), \quad \|x\| \ge 0\; ,
\end{equation}
where $\sigma^2 > 0$ is called the partial sill or spatial variance, $\phi > 0$ is the {scale or decay parameter}, $\nu > 0$ is a smoothness parameter, $\Gamma(\cdot)$ is the Gamma function, and $K_{\nu}(\cdot)$ is the modified Bessel function of the second kind of order $\nu$ \cite[Section 10]{AS65}. The corresponding spectral density is
\begin{equation}
\label{eq:spectralb}
f_{K_w}(u) = C \frac{\sigma^2 \phi^{2 \nu}}{(\phi^2 + u^2)^{\nu + d/2}} \quad \mbox{for some } C >0.
\end{equation}
When $\nu = 1/2$, the covariogram \eqref{eq:Matern} simplifies to the exponential (Ornstein-Uhlenbeck in one-dimension) kernel 
$$K_w(x; \, \sigma^2, \phi): = \sigma^2 \exp(- \phi \|x\|).$$
For the measurement error, we assume $\{\epsilon(s) : s \in {\cal S} \subset \mathbb{R}^d\}$ is Gaussian white noise with covariogram $K_{\epsilon}(y; \, \tau^2): =\tau^2 \delta_0$, where $\delta_0$ is the indicator function at $0$ and $\tau^2$ is the nugget. 
The processes $\{w(s), \, s \in D \subset \mathbb{R}^d\}$ and $\{\epsilon(s), \, s \in D \subset \mathbb{R}^d\}$ are independent. Hence, a Mat\'ern model with measurement error is a stationary Gaussian process with covariogram 
\begin{equation}
\label{eq:Maternmeas}
K(x; \, \tau^2, \sigma^2, \phi, \nu): = K_w(x; \,  \sigma^2, \phi, \nu) + K_{\epsilon}(x; \, \tau^2).
\end{equation}

Our approach will depend upon identifying microergodic parameters in the above model. The remainder of the article evolves as follows.  We review the discussion in \cite{Zhang04} for the Mat\'ern model with measurement error, claiming that only  $\theta = \{\sigma^2 \phi^{2 \nu},\tau^2\}$ can have infill consistent estimators when $d \leq 3$. Subsequently, we establish that the maximum likelihood estimates for $\theta$ are consistent and are asymptotically normal. This extends the main results in \cite{CSY00} to the case with dimension $d \leq 3$. The asymptotic properties of interpolation are explored mainly through simulations, and we demonstrate the role of $\theta$ in interpolation. We conclude with some insights and directions for future work.

\section{Asymptotic theory for estimation and prediction}\label{sec: asymp_thm}

\subsection{Identifiability}\label{subsec: identifiability}
\cite{Zhang04} showed that for the Mat\'ern model without measurement error, when fixing the smoothness parameter $\nu > 0$ and $d \leq 3$,  there are no (weakly) infill consistent estimators for either the partial sill $\sigma^2$ or the scale parameter $\phi$. Such results rely upon the equivalence and orthogonality of Gaussian measures. Two probability measures $P_1$ and $P_2$ on a measurable space $(\Omega, {\cal F})$ are said to be equivalent, denoted $P_1 \equiv P_2$, if they are absolutely continuous with respect to each other. Thus, $P_1 \equiv P_2$ implies that for all $A\in {\cal F}$, $P_1(A)=0$ if and only if $P_2(A)=0$. On the other hand, $P_1$ and $P_2$ are orthogonal, denoted $P_1 \perp P_2$, if there exists $A\in {\cal F}$ for which $P_1(A)=1$ and $P_2(A)=0$. While measures may be neither equivalent nor orthogonal, Gaussian measures are one or the other. For a Gaussian probability measure $P_{\theta}$ indexed by a set of parameters $\theta$, we say that $\theta$ is microergodic if $P_{\theta_1}\equiv P_{\theta_2}$ if and only if $\theta_1 = \theta_2$. For further background, see Chapter~6 
in \cite{Stein99} and 
\cite{zhang2012asymptotics}. Furthermore, two Gaussian probability measures defined by Mat\'ern covariograms $K_w(\cdot; \, \sigma_1^2, \phi_1, \nu)$ and $K_w(\cdot; \, \sigma_2^2, \phi_2, \nu)$ are equivalent if and only if $\sigma_1^2 \phi_1^{2 \nu} = \sigma_2^2 \phi_2^{2 \nu}$ 
\citep[Theorem~2 in ][]{Zhang04} and, consequently, one cannot consistently estimate $\sigma^2$ or $\phi$ in the Mat\'ern model \eqref{eq:Matern} \citep[Corollary~1 in][]{Zhang04}. 

We first characterise identifiability for the Mat\'ern model with measurement error, i.e., with covariogram given by \eqref{eq:Maternmeas}.
Over a closed set $S \subset \mathbb{R}^d$, let $G_S(m,K)$ denote the Gaussian measure of the random field on $S$ with mean function $m$ and covariance function $K$. Consider two different specifications for $w(s)$ in \eqref{eq:basic} corresponding to mean $m_i$ and covariogram $K_i$ for $i=1,2$. The respective measures on the realizations of $w(s)$ over $S$ will be denoted by $G_{S}(m_i, K_i)$ for $i=1,2$. If ${\chi}= \{s_1, s_2, \ldots\} $ is a sequence of points in $S$, then the probability measure for the sequence of outcomes over $\chi$, i.e., $\{y(s_j) : s_j\in \chi\}$, is denoted $G_{\chi}(m_i, K_i, \tau_i^2)$ under model $i$. 
The following lemma is familiar. 
\begin{Lemma} \label{lem:SteinThm6}
Let $S$ be a closed set, $w(s)$ be a mean square continuous process on $S$ under $G_S(m_1, K_1)$, and ${\chi}$ be a dense sequence of points in $S$. Then, (i) if $\tau_1^2 \ne \tau_2^2$, then $G_{\chi}(m_1, K_1, \tau_1^2) \perp G_{\chi}(m_2, K_2, \tau_2^2)$; and (ii) if $\tau_1^2 = \tau_2^2$, then $G_{\chi}(m_1, K_1, \tau_1^2) \equiv G_{ \chi}(m_2, K_2, \tau_2^2)$ if and only if $G_S(m_1, K_1) \equiv G_S(m_2, K_2)$.
\end{Lemma}
\begin{proof}
	See Theorem 6 in Chapter 4 of \cite{Stein99}.
\end{proof}



{According to \citet[p.121]{Stein99} (or \cite[III.4.1]{IR78}), two Gaussian measures $G_S(m, K) \equiv G_S(0, K)$ if and only if the mean function $m(\cdot)$ can be extended to a square-integrable function on $\mathbb{R}^d$ whose Fourier transform $\widehat{m}(\omega)$ satisfies $\displaystyle \int_{\mathbb{R}^d} \frac{|\widehat{m}(\omega)|^2}{f_K(\omega)} d\omega < \infty$, where $f_K$ denotes the spectral density of the covariance function $K$. In such a situation, the mean function $m(\cdot)$ of the Gaussian process is not identifiable. A specific example is the Gaussian measure with $m(x) = \beta^\top x$, where $\beta \in \mathbb{R}^d$ and $K$ is the Mat\'ern covariogram. From a practical inferential standpoint, most of the insights obtained from the subsequent theoretical developments will apply to de-trended processes.} {The following result adapts} Lemma~\ref{lem:SteinThm6} to the Mat\'ern model with measurement error {and} summarizes the identifiability issue with measurement error.
\begin{theorem}\label{thm:identify}
	Let $S \subset \mathbb{R}^d$ be a compact set. For $i = 1,2$, let $P_i$ be the probability measure of the Gaussian process on $S$ with mean zero and covariance $K(\cdot; \tau_i^2, \sigma_i^2, \phi_i, \nu)$ defined by \eqref{eq:Maternmeas}. Then, (i) if $\tau_1^2 \ne \tau_2^2$, then $P_1 \perp P_2$; and (ii) if $\tau_1^2 = \tau_2^2$, then for $d \le 3$, $P_1 \equiv P_2$ if and only if $\sigma_1^2 \phi_1^{2 \nu} = \sigma_2^2 \phi_2^{2 \nu}$, and for $d \ge 5$, $P_1 \equiv P_2$ if and only if $(\sigma_1^2, \phi_1) = (\sigma_2^2, \phi_2)$.
\end{theorem}




\begin{proof}
	Denote $K_i$ 
	for $K_{w}(\cdot; \, \sigma_i^2, \phi_i, \nu)$
	.
	It is easy to see that $w(s)$ is mean square continuous on $S$ under $G_S(0, K_i)$. From Lemma~\ref{lem:SteinThm6}, we know that if $\tau_1^2 \ne \tau_2^2$, for any dense sequence ${\chi}$, $G_{\chi}(0, K_1, \tau_1^2) \perp G_{\chi}(0, K_2, \tau_2^2)$. Therefore, $P_1 \perp P_2$. This proves (i).
	
	Next, suppose $\tau_1^2 = \tau_2^2$. From Theorem~2 in \cite{Zhang04}, we know that for $d\leq 3$ $G_S(0, K_1) \equiv G_S(0, K_2)$ if and only if $\sigma_1^2 \phi_1^{2 \nu} = \sigma_2^2 \phi_2^{2 \nu}$. Corollary~3 in \cite{A10} shows that, for $d \ge 5$, $G_S(0, K_1) \perp G_S(0, K_2)$ if $\{\sigma_1^2, \phi_1\} \ne \{\sigma_2^2, \phi_2\}$. 
	A straightforward application of Lemma~\ref{lem:SteinThm6} proves (ii).
\end{proof}



{
Theorem \ref{thm:identify} characterizes equivalence and orthogonality of Mat\'ern based Gaussian measures in terms of their parameters. Here it is instructive to distinguish between $d \leq 3$ and $d \geq 5$. The results in \cite{Zhang04} emerge as special cases when $\tau_1^2 = \tau_2^2 = 0$ and $\sigma_1^2 \phi_1^{2 \nu} = \sigma_2^2 \phi_2^{2 \nu}$ for $d \le 3$. Combining Theorem~\ref{thm:identify} with the argument provided in Corollary~1 of \cite{Zhang04}, we can conclude that $\sigma^2$ and $\phi$ are not consistently estimable. We provide this as an immediate corollary to Theorem~\ref{thm:identify}.  

\begin{Corollary}
	Let $y(s)$, $s \in S \subset \mathbb{R}^d, d \leq 3$ be a Gaussian process with a covariogram as in \eqref{eq:Maternmeas}, and $S_n$, $n \ge 1$ be an increasing sequence of subsets of $S$. Given observations of $y(s)$, $s \in S_n$, there do not exist estimates $\widehat{\sigma}_n^2$ and $\widehat{\phi}_n$ that are consistent.
\end{Corollary}

Consequently, the joint maximum likelihood estimators of $\{\sigma^2, \phi\}$ are not consistent estimators. In contrast to $\{\sigma^2, \phi\}$, we show in Theorem~\ref{thm:consistency} that the maximum likelihood estimator of the nugget $\tau^2$ is consistent.

Turning to $d \ge 5$, it follows from Theorem \ref{thm:identify} that there exist joint estimates $(\widehat{\tau}_n^2, \widehat{\sigma}_n^2, \widehat{\phi}_n)$ which converge to $(\tau^2, \sigma^2, \phi)$. For instance, letting $t_{i,n} = \frac{i}{n} 1_d$, where $1_d$ denotes the vector of $1$'s in $\mathbb{R}^d$, we can take $\widehat{\tau}_n^2 = \frac{1}{2|\mathcal{I}|} \sum_{i \in \mathcal{I}} (y(t_{i+1,n}) - y(t_{i,n}))^2$, 
where $\mathcal{I} = \{i \in \mathbb{Z}: t_{i+1, n}, t_{i, n} \in S\}$ and $|\mathcal{I}|$ is the cardinality of $\mathcal{I}$. Further, \cite{A10} constructed consistent estimators of $(\sigma^2, \phi)$ based on higher order increments of $y$. However, it is currently unknown whether the joint maximum likelihood estimators of $(\tau^2, \sigma^2, \phi)$ are consistent. Even for the Mat\'ern model without a nugget ($\tau^2 = 0$), the consistency of the joint maximum likelihood estimators of $(\sigma^2, \phi)$ remains unresolved.

The characterization of equivalence and orthogonality of $P_1$ and $P_2$ is also open in the critical dimension $d = 4$.  The balance of this paper focuses on the asymptotic properties of the maximum likelihood estimates and predictions for the Mat\'ern model with nugget when $d \le 3$ with additional discussions and results for $d \ge 5$ in Section~\ref{subsec: highd}.}


\subsection{Parameter estimation}\label{subsec: estimation}
Theorem \ref{thm:identify} implies that if $\nu$ is fixed in the specification of $w(s)$ in (\ref{eq:basic}), then $\sigma^2 \phi^{2 \nu}$ and the nugget $\tau^2$ will be identifiable. In view of this, we consider the estimation of the microergodic parameter $\kappa := \sigma^2 \phi^{2 \nu}$ and the nugget $\tau^2$ with fixed decay $\phi$. Our main results concern the consistency and the asymptotic normality of the maximum likelihood estimators of $\kappa$ and $\tau^2$ when the observations are taken from $y(\cdot)$ modelled by \eqref{eq:basic}.  

To proceed further, we need some notations. Let $\chi_n = \{s_1, \ldots, s_n\}$ be the sampled points in $S$, $y_i : = y(s_i)$, $i = 1, \ldots, n$ be the corresponding observations, and let $\displaystyle K_n: = \left\{ K_w(s_i - s_j; \, \sigma^2, \phi, \nu) \right\}_{1 \le i,j \le n}$
denote the $n \times n$ Mat\'ern covariance matrix over locations $\chi_n$. Let $\{\lambda_{i}^{(n)},\; i = 1, \ldots, n\}$ be the eigenvalues of ${\frac{1}{\sigma^2}}K_n$ in decreasing order. The covariance matrix of the observations $y = (y_1, \ldots, y_n)^\top$ is $V_n = \tau^2 I_n + K_n$,  
the likelihood is denoted by $\mathcal{L}(\tau^2, \sigma^2, \phi)$, and the (rescaled) negative log-likelihood is 
\begin{equation}\label{eq: neg_loglik}
\ell(\tau^2, \sigma^2, \phi) : = \log \det V_n + y^\top V_n^{-1} y.
\end{equation}
Let $\{\sigma_0^2, \phi_0, \tau_0^2\}$ be the true generating values of $\{\sigma^2, \phi, \tau^2\}$, $\kappa_0 = \sigma_0^2\phi_0^{2\nu}$. Assume that the smoothness parameter $\nu > 0$ is known. For any fixed $\phi_1 > 0$, let $(\widehat{\tau}^2_n(\phi_1), \widehat{\sigma}^2_n(\phi_1))$ be the maximum likelihood estimators of $\{\tau^2, \sigma^2\}$.  That is,
\begin{align}\label{eq:MLE}
(\widehat{\tau}^2_n(\phi_1), \widehat{\sigma}^2_n(\phi_1)) : & = \argmax_{(\tau^2, \sigma^2) \in D} \mathcal{L}(\tau^2, \sigma^2, \phi_1) = \argmin_{(\tau^2, \sigma^2) \in D} \ell(\tau^2, \sigma^2, \phi_1)
\end{align}
where $D = [a, b] \times [c,d]$ with $0 < a < b < \infty$ and $0 < c < d < \infty$. 
To simplify notations, write $\widehat{\tau}_n^2$, $\widehat{\sigma}_n^2$ for $\widehat{\tau}^2_n(\phi_1)$, $\widehat{\sigma}^2_n(\phi_1)$. 
Unlike the Mat\'ern model \eqref{eq:Matern}, there is no explicit formula for $\widehat{\tau}_n^2$ and $\widehat{\sigma}_n^2$ in the Mat\'ern model with measurement error. Another difficulty of the analysis is that $\mathcal{L}$ is not concave, so the (rescaled) negative log-likelihood $\ell(\tau^2, \sigma^2, \phi_1)$ may have local minima and stationary points. 
Nevertheless, we are able to establish the theorems regarding the consistency and asymptotic normality at these stationary points under some assumptions of the eigenvalue asymptotics.

\subsubsection{Eigenvalue decay}
We first give an upper bound for the eigenvalues $\lambda^{(n)}_i$, which is of independent interest.
The argument we provide below works for a large class of covariograms, including the Mat\'ern model.
In the sequel, the symbol $\asymp$ indicates asymptotically bounded from below and above. 
We follow closely the presentation of \cite{B18}.
Let $\Omega$ be a domain of $\mathbb{R}^d$, and $K(\cdot)$ be a positive definite radial basis kernel on $\mathbb{R}^d$.
Denote $\mathcal{H}$ to be the Reproducing Kernel Hilbert Space corresponding to the kernel $K$, which is also the native space associated to the kernel $K$.
Given a probability measure $\mu$ on $\Omega$, define the integral operator $\mathcal{K}_{\mu}: L^2_{\mu} \to L^2_{\mu}$ by 
\begin{equation*}
\mathcal{K}_{\mu}f(x):= \int_{\Omega} K(x - z) f(z) \mu(dz).
\end{equation*}
In particular, if $\mu = \frac{1}{n} \sum_{i = 1}^{n}\delta_{s_i}$, $\mathcal{K}_\mu$ corresponds to the kernel matrix $ \{\frac{1}{n} K(s_i - s_j)\}_{1\leq i,j \leq n}$.
It is well known that $\mathcal{K}_{\mu}f \in \mathcal{H}$ for $f \in L^2_{\mu}$, and any function in $\mathcal{H}$ induces a function in $L^2_{\mu}$ by restricting it to the support of $\mu$ (see Section 2 of \cite{B18}).
Call $\mathcal{R}_{\mu}: \mathcal{H} \to L^2_{\mu}$ the restriction operator.

The key idea of \cite{B18} is to get a measure-independent upper bound for the eigenvalues of $\mathcal{K}_{\mu}$ for infinitely smooth kernels, while the argument can carry over to kernels with limited smoothness; that is, the spectral density of $K$ satisfies $f(u) \asymp u^{-\beta - d}$ ($\beta$-smooth).
By \eqref{eq:spectralb}, the Mat\'ern covariogram is $2\nu$-smooth. 
Given $\chi = \{s_1, \ldots, s_n\} \subset \Omega$, 
let $S_{\chi}: \mathcal{H} \to \mathcal{H}$ be the interpolation operator defined by 
\begin{equation*}
S_{\chi}f(x) = \sum_{i = 1}^n \alpha_i K(x_i - x),
\end{equation*}
where $(\alpha_1, \ldots, \alpha_n)^\top = K_n^{-1}(f(x_1), \ldots, f(x_n))^\top$ with $K_n = \{K(s_i - s_j)\}_{1 \le i,j \le n}$.
By letting $h = \max_{s \in S}\min_{1 \leq i \leq n}  \|s - s_i\|$, 
\citet[p985]{SS16} proved that there exists $C > 0$ (independent of $n$) such that
\begin{equation}\label{approx_dist}
\|\mathcal{R}_{\mu} - S_{\chi} \|_{\mathcal{H} \to L^2_{\mu}} \leq C h^{(\beta + d) / 2}.
\end{equation}
Here $\|\cdot\|_{\mathcal{H} \to L^2_{\mu}}$ denotes the operator norm. So \eqref{approx_dist} is a limited smoothness version of \citet[Theorem A]{B18}. 
The following result is adapted from Theorem~1 in \citet{B18} to the $\beta$-smooth kernel.
\begin{theorem} \label{thm:BT1}
	Suppose $\mathcal{T}: V \to \mathcal{H}$ is a map from a Banach space $V$ to a Reproducing Kernel Hilbert Space of functions on $\mathbb{R}^d$, $\mathcal{H}$ corresponding to a $\beta$-smooth radial basis kernel. Then there exists a map $\mathcal{T}_n$ from $V$ to an $n$-dimensional linear subspace $\mathcal{H}_n \subset \mathcal{H}$, such that
	$$
	\|\mathcal{T} - \mathcal{T}_n\|_{V \to L_\mu^2} \leq C \|\mathcal{T}\|_{V\to \mathcal{H}} \, n^{- \frac{\beta + d}{d}}
	$$
	for $C > 0$ indepedent of $\mathcal{T}$ and $\mu$. Moreover, (1) the subspace $\mathcal{H}_n$ is independent of $\mathcal{T}$; (2) if $\mathcal{T}$ is linear operator, $\mathcal{T}_n$ is also a linear operator.
\end{theorem}
\begin{proof}
	The proof follows immediately from Theorem~1 in \citet{B18} {by substituting Theorem A therein with the bound in \eqref{approx_dist}.}
\end{proof}

The following theorem is adapted from Theorem $2$ in \cite{B18} for $\beta$-smooth kernels.
\begin{theorem} \label{thm:BT2}
	Let $K$ be a $\beta$-smooth radial basis kernel, and $\lambda_i(\mathcal{K_{\mu}})$ be the $i^{th}$ largest eigenvalue of $\mathcal{K}_{\mu}$. Then there exists $C > 0$ such that 
	$$\lambda_i(\mathcal{K}_\mu) \le C i^{- \frac{\beta + d}{d}}.$$
\end{theorem}
\begin{proof}
	The proof follows by combining Theorem~\ref{thm:BT1} above with Lemma~1 in \citet{B18}.
\end{proof}

\begin{Corollary}\label{thm:eigen_upper}
	Assume that 
	$\max_{s \in S} \min_{1 \le i \le n} \|s - s_i\| \asymp n^{-1/d}$.
	There exists $C > 0$ independent of $n$ such that 
	\begin{equation}\label{lambda_upper_B}
	\lambda_i^{(n)} \leq C n i^{-2\nu/d -1} \quad  \mbox{ for all } \,  i = 1, \ldots, n.
	\end{equation}
	\end{Corollary}
\begin{proof}
	This follows immediately from applying Theorem~\ref{thm:BT2} with $\mu = \frac{1}{n} \sum_{i = 1}^{n}\delta_{s_i}$ and $\beta = 2 \nu$.
\end{proof}

{
Here, it is natural to enquire about a matching lower-bound for the eigenvalues $\lambda^{(n)}_i$ under a suitable condition on the sampled point locations. To develop a rigorous framework, we lay down the following assumptions and provide heuristics and numerical evidence to show why these assumptions are expected to be true.}

\begin{assumption}\label{assump:1}
	Assume that $\min_{1 \le i \ne j \le n} ||s_i - s_j|| \asymp n^{-1/d}$. 
	There exists $c >0$ such that 
	\begin{equation}\label{lambda_lower_B}
	\lambda_i^{(n)} \geq c n i^{-2\nu/d -1} \quad  \mbox{ for all } \,  i = 1, \ldots, n.
	\end{equation}
\end{assumption}

The lower bound \eqref{lambda_lower_B} holds for the largest eigenvalues. A lesser known result of \cite{Schaback95} shows that \eqref{lambda_lower_B} also holds for {
the smallest eigenvalues with} $i \asymp n$.
{
However, there is no rigorous result for the lower bound of eigenvalues in full generality.}
Particularly interesting cases are $i \asymp n^{\alpha}$ for $0 < \alpha < 1$, which leave the lower bound \eqref{lambda_lower_B} open.
{
In Figure~\ref{fig:assump}, we plot the values of $\lambda_i^{(n)} / (n i^{-2\nu-1})$ with sampled points on the regular grid $[0,1) \cap n^{-1} \mathbb{Z}$ for $\nu = 0.9, 1.5$, $n$ ranging from $100$ to $3000$, and $i = n^{0.5}$, $n^{0.75}$, $n^{0.9}$. 
Consistent with Assumption~\ref{assump:1}, the profile plots of $\lambda_i^{(n)} / (n i^{-2\nu/d-1})$ get flat as $n$ increases. Furthermore, we see that when the points are sampled on $[0,1) \cap n^{-1} \mathbb{Z}$, the quantity $\lambda_i^{(n)} / (n i^{-2\nu-1})$ tends to converge as $n, i$ become large. This observation leads to the following stronger conjecture.
}
\begin{figure}[h]
	\centering
	\includegraphics[width = 0.9\linewidth]{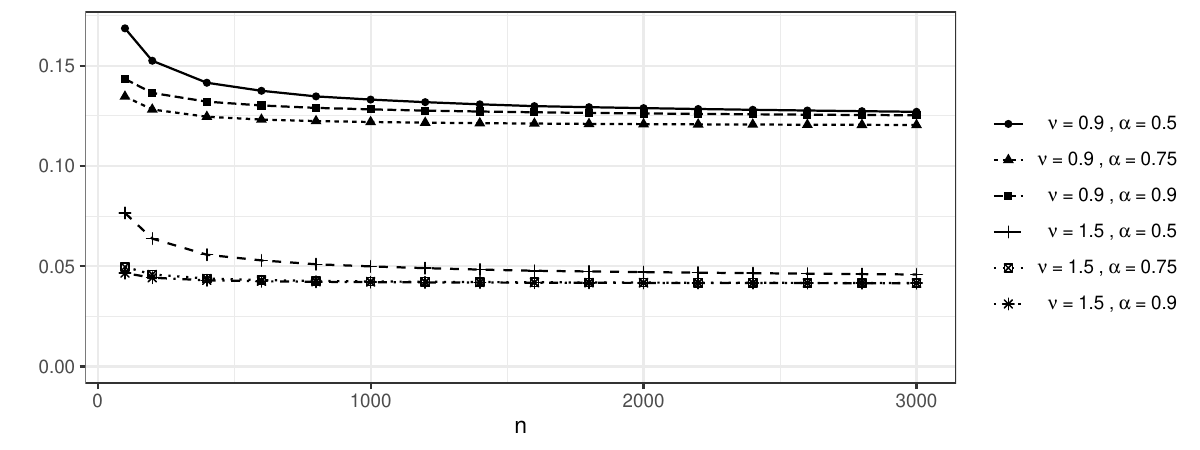}
	\caption{Trend of $\lambda_i^{(n)} / (n i^{-2\nu/d-1})$ for $i = n^\alpha$ when the points are sampled on the regular grid $[0,1) \cap n^{-1} \mathbb{Z}$. Parameters $\phi$ and $\sigma^2$ in Mat\'ern covariogram are set to be $1.0$ and $1.0$, respectively.
	\label{fig:assump}}
\end{figure}

\begin{assumption}\label{assump:2}
	Let $\chi_n= [0,1)^d \cap n^{-1/d} \mathbb{Z}^d$ be the regular grid.
	There exists $A = A(\phi, \nu, d)> 0$ such that
	\begin{equation}\label{lambda_lim}
	\lambda_i^{(n)} / (n i^{-2\nu/d-1}) \to A \quad \mbox{ as } \, n, i \to \infty
	\end{equation}
\end{assumption}

{
Besides the numerical evidence, let us explain heuristics underlying this assumption from a theoretical viewpoint. First, Assumption \ref{assump:2} has been rigorously proved in \cite{CSY00} for Ornstein–Uhlenbeck processes, corresponding to the case of $\nu = 1/2$ and $d = 1$.}
Furthermore, for the regular grid $\chi_n$, 
the scaled covariance matrix $\frac{1}{n {\sigma^2}}K_n$ is viewed as the discretization of the integral operator
\begin{equation*}
\mathcal{K}f(x) : = \int_{[0,1]^d} K_{w}(s-t; {1}, \phi, \nu) f(t)dt,
\end{equation*}
where $f$ is a test function.
The integral operator $\mathcal{K}$ has eigenvalues $\lambda_1 \ge \lambda_2 \ge \cdots > 0$.
Intuitively, $\lambda^{(n)}_i/n \approx \lambda_i$ which is at least true for fixed $i$. 
\citet{SS16} observed that $\lambda_i = h_{i-1}^2$, with ${h_i}$ the $i$-width of the unit Sobolev ball in the $L^2$ space. 
Using a differential operator approach, \citet{J72} showed that $\lim_{i \rightarrow \infty} i^{\frac{2\nu +d}{2d}} {h_i} = C'$.
The above two results imply that $\lim_{i \rightarrow \infty} i^{2 \nu/d + 1} \lambda_i = C'^2$.
{
Thus, we expect that $\lambda^{(n)}_i/(n i^{-2\nu/d-1}) \approx i^{2\nu/d-1}\lambda_i \approx C'^2$ as $n,i \rightarrow \infty$, though the error $|\lambda^{(n)}_i/n- \lambda_i|$ is not easy to estimate.}

To proceed further, we need the following lemma which is proved by elementary calculus.
\begin{Lemma}\label{lemma: ab_asymp_B}
	Assume that $\max_{s \in S} \min_{1 \le i \le n} \|s - s_i\| \asymp n^{-1/d}$ and $\min_{1 \le i \ne j \le n} ||s_i - s_j|| \asymp n^{-1/d}$.
	Let $a_{ni} = 1 / (\widehat{\tau}^2_n + \widehat{\sigma}_n^2 \lambda_i^{(n)} )$, 
	$b_{ni} = \lambda_i^{(n)} / (\tau^2_0 + \widehat{\sigma}^{2}_n \lambda_i^{(n)})$,
	$a^0_{ni} = 1 / (\tau_0^2 + \sigma^2 \lambda_i^{(n)} )$,
	and $b^0_{ni} = \lambda_i^{(n)} a_{ni}^0$.

	(1) There exists $C > 0$ such that 
	$$
	\sum_{i = 1}^n a^2_{ni} \asymp n\;, \;\sum_{i = 1}^n \lambda_i^{(n)} a^2_{ni} \le C n^\frac{1}{2\nu/d + 1}\;, \; \sum_{i = 1}^n b_{ni} \le C n^\frac{1}{2\nu/d + 1}\; , \; \sum_{i = 1}^n b_{ni}^2 \le C n^\frac{1}{2\nu/d + 1}\;.$$
	
	(2) Under Assumption \ref{assump:1}, 
	$$
	 \sum_{i = 1}^n \lambda_i^{(n)} a^2_{ni} \asymp n^\frac{1}{2\nu/d + 1}\;, \; \sum_{i = 1}^n b_{ni} \asymp n^\frac{1}{2\nu/d + 1}\; , \; \sum_{i = 1}^n b_{ni}^2 \asymp n^\frac{1}{2\nu/d + 1}\;.$$
	
	(3) Under Assumption \ref{assump:2}, there exist $c_1(\sigma), c_2(\sigma), c_3(\sigma) > 0$ such that as $n \to \infty$,
	\begin{equation*}
	\frac{1}{n} \sum_{i = 1}^n (a_{ni}^0 )^2 \to c_1(\sigma), \quad
	\frac{1}{n} \sum_{i = 1}^n (a_{ni}^0 )^4 \to c_2(\sigma), \quad 
	\frac{1}{n^{1/(1+ 2 \nu/d)}} \sum_{i = 1}^{n} (b_{ni}^0)^2 \to c_3(\sigma).
	\end{equation*}
\end{Lemma}

\subsubsection{Consistency of the maximum likelihood estimator}\label{subsubsec: consistency_mle}
{
{We begin our development of} the consistency of the maximum likelihood estimator of the nugget $\widehat{\tau}^2_n$ and the microergodic parameter $\widehat{\sigma}_n^2 \phi_1^{2 \nu}$ under Assumption~\ref{assump:1}. 
{We point out that} the consistency of the nugget $\widehat{\tau}^2_n$ is true without Assumption \ref{assump:1} on the lower bound for eigenvalues,
and $\widehat{\tau}^2_n$ remains consistent even when $\sigma^2$ and $\phi$ are misspecified.
}

\begin{theorem}\label{thm:consistency}
Assume that $(\tau_0^2, \sigma_0^2) \in D$, $\chi_n: = \{s_1, \ldots, s_n\}$ satisfy
	\[
	\max_{s \in S} \min_{1 \le i \le n} \|s - s_i\| \asymp n^{-1/d} \quad \mbox{and} \quad \min_{1 \le i \ne j \le n} ||s_i - s_j|| \asymp n^{-1/d}.
	\]
Let $P_0$ be the probability measure of the Mat\'ern model with covariogram $K(\cdot; \, \tau_0^2, \sigma_0^2, \phi_0, \nu)$. Then $\widehat{\tau}_n^2 \rightarrow \tau_0^2$ almost surely under $P_0$. Further assume that the conditions in Assumption~\ref{assump:1} hold.
Then $\widehat{\sigma}_n^2 \phi_1^{2 \nu} \rightarrow \kappa_0$ almost surely under $P_0$.
\end{theorem}

\begin{proof}
	Let $P_1$ be the probability measure corresponding to 
	$K(\cdot; \, \tau_0^2, \sigma_1^2, \phi_1, \nu)$, where $\sigma_1^{2}: = \kappa_0 / \phi_1^{2 \nu}$.
	We first prove that $\widehat{\tau}^2_n \rightarrow \tau^2_0$ almost surely under $P_0$. From  Theorem~\ref{thm:identify} we know that $P_0 \equiv P_1$. Hence, it suffices to prove that $\widehat{\tau}^2_n \rightarrow \tau^2_0$ almost surely under $P_1$.
	Under $P_1$, we can rewrite \eqref{eq: neg_loglik} as
	\begin{equation}
	\label{eq:ellkappa}
	\ell(\tau^2, \widehat{\sigma}_n^2, \phi_1) = \sum_{i = 1}^n \frac{\tau^2_0 + \sigma_1^2 \lambda_i^{(n)}}{\tau^2 + \widehat{\sigma}_n^2 \lambda_i^{(n)}} W_i^2 + \sum_{i=1}^n \log(\tau^2 + \widehat{\sigma}_n^2 \lambda_i^{(n)}),
	\end{equation}
	where $W_i \stackrel{iid}{\sim} \mathcal{N}(0,1)$. 
	The maximum likelihood estimator $\widehat{\tau}_n^2$ of $\tau^2$ satisfies
	\begin{equation}
	\label{eq:kappaidentity}
	(\tau^2_0 - \widehat{\tau}^2_n) \cdot \sum_{i=1}^n W_i^2 a_{ni}^2 = \sum_{i = 1}^n \widehat{\tau}^2_n(1-W_i^2)a_{ni}^2 + \sum_{i= 1}^n (\widehat{\sigma}_n^2 - \sigma_1^2 W_i^2) \lambda_i^{(n)} a_{ni}^2.
	\end{equation}
	where $a_{ni} = 1 / (\widehat{\tau}^2_n + \widehat{\sigma}_n^2 \lambda_i^{(n)} )$. 
	By Lemma \ref{lemma: ab_asymp_B} (1), we have
	$\sum_{i = 1}^n a_{ni}^2 \asymp n$ and $\sum_{i = 1}^n \lambda_i^{(n)} a_{ni}^2 \le C n^{1/(2 \nu/d + 1)}$ for some $C > 0$. 
	Using the results of \cite{E06}, we obtain
	\begin{equation}\label{eq:convest}
	\frac{\sum_{i=1}^n W_i^2 a_{ni}^2}{\sum_{i = 1}^n a_{ni}^2} \rightarrow 1, \quad
	\frac{\sum_{i = 1}^n \widehat{\tau}^2_n (1-W_i^2)a_{ni}^2}{\sum_{i = 1}^n a_{ni}^2} \rightarrow 0 \quad \mbox{and} \quad 
	\frac{\sum_{i= 1}^n (\widehat{\sigma}_n^2 - \sigma_1^2 W_i^2) \lambda_i^{(n)} a_{ni}^2}{\sum_{i = 1}^n a_{ni}^2} \rightarrow 0.
	\end{equation}
	Here we also give an elementary proof of the last convergence in \eqref{eq:convest}. 
 It is clear that 
\begin{equation*}
\mbox{l.h.s} = \frac{\sum_{i = 1}^n \widehat{\sigma}_n^2 \lambda_i^{(n)} a_{ni}^2}{\sum_{i =1}^n a_{ni}^2} - \sigma_1^2 \frac{\sum_{i = 1}^n \lambda_i^{(n)} a_{ni}^2 W_i^2}{\sum_{i = 1}^n a_{ni}^2}: = (a) - (b).
\end{equation*}
From \eqref{eq:MLE} we know that $\widehat{\sigma}_n^2$ is bounded from above. 
So the term $(a) \lesssim \frac{\sum_{i = 1}^n \lambda_i^{(n)} a_{ni}^2}{\sum_{i =1}^n a_{ni}^2} \to 0$ by Lemma \ref{lemma: ab_asymp_B} (1). 
Next
\begin{equation*}
Z_{ni}:=\lambda_i^{(n)} a_{ni}^2 = \frac{\lambda_i^{(n)}}{(\widehat{\tau}_n^2 + \widehat{\sigma}_n^2 \lambda_i^{(n)})^2} 
\le \frac{\lambda_i^{(n)}}{4 \widehat{\tau}_n^2 \widehat{\sigma}_n^2 \lambda_i^{(n)}} \leq C,
\end{equation*}
again, from \eqref{eq:MLE} we note that $\widehat{\tau}_n^2$ and $\widehat{\sigma}_n^2$ are bounded away from $0$. 
Moreover, by Lemma \ref{lemma: ab_asymp_B} (1), 
$\sum_{i = 1}^n Z_{ni} \lesssim n^\theta$ for some $\theta < 1$
and so $\sum_{i = 1}^n Z_{ni}^2 \lesssim n^\theta$
(since $Z_{ni} \le C$).

Let $S_n: = \sum_{i = 1}^n \frac{Z_{ni} (W_{i}^2 - 1)}{n}$.
Fix $\delta > 0$. 
We have
\begin{align*}
& \qquad \mathbb{P}\left(\sup_{m+1
\leq k \leq n} |S_k - S_m| \ge \delta \right) \le 
\mathbb{P}\left( \sup_{m+1 \le k \le n} \left| \sum_{i = 1}^m  \left(\frac{Z_{ki}}{k} - \frac{Z_{mi}}{m}\right) (W_i^2 -1) \right| \ge \frac{99}{100} \delta \right) \\
& \qquad \qquad \qquad \qquad \qquad \qquad \qquad+ \mathbb{P}\left( \sup_{m+1 \le k \le n} \left|\sum_{i = m+1}^k \frac{Z_{ki}}{k} (W_i^2 -1) \right| > \frac{\delta}{100} \right): = (c) +(d).
\end{align*}
For the term $(c)$, we get for some $C' > 0$,
\begin{align*}
    (c) & \le \mathbb{P}\left( \left|\sum_{i = 1}^m \frac{Z_{mi}}{m} (W_i^2 - 1) \right| \ge \frac{99}{200} \delta \right) + \mathbb{P}\left(\sup_{m+1 \le k \le n} \left|\sum_{i = 1}^m \frac{Z_{ki}}{k} (W_i^2 - 1) \right| \ge \frac{99}{200} \delta \right) \\
    & \le  \mathbb{P}\left( \left|\sum_{i = 1}^m \frac{Z_{mi}}{m} (W_i^2 - 1) \right| \ge \frac{99}{200} \delta \right) + \sum_{k = m+1}^n \mathbb{P}\left(\left|\sum_{i = 1}^m \frac{Z_{ki}}{k} (W_i^2 - 1) \right| \ge \frac{99}{200} \delta \right) \\
    & \lesssim \frac{1}{\delta^2 m^{2 - \theta}} + \sum_{k = m+1}^n \frac{1}{\delta^2 k^{2 - \theta}} \le \frac{C'}{\delta^2 m^{1-\theta}}.
\end{align*}
For the term $(d)$, we have for some $C'' > 0$, 
\begin{align*}
    (d) & \le \sum_{k = m+1}^n \mathbb{P} \left(\left|\sum_{i = m+1}^k \frac{Z_{ki}}{k} (W_i^2 -1) \right| > \frac{\delta}{100} \right) \lesssim \sum_{k = m+1}^n \frac{1}{\delta^2 k^{2 - \theta}} \le \frac{C''}{\delta^2 m^{1 - \theta}}.
\end{align*}
Combining the above estimates and passing $n \to \infty$ yield $\mathbb{P}\left(\sup_{n \ge m} |S_n - S_m| \ge \delta \right) \to 0$ as $m \to \infty$.
It follows that $S_n$ converges almost surely (see the proof of Theorem 2.5.6 in \cite{Durrett}). 
Since $\sum_{i = 1}^n Z_{ni}/n \to 0$ by Lemma \ref{lemma: ab_asymp_B} (1), we have $\sum_{i = 1}^n Z_{ni}W_i^2/n$ converges almost surely to some $U \ge 0$ almost surely. 
By the boundedness of $Z_{ni}$, we get
$\sup_n \mathbb{E}\left[\left(\sum_{i = 1}^n Z_{ni}W_i^2/n \right)^2\right] < \infty$
so $\sum_{i = 1}^n Z_{ni}W_i^2/n$ is uniformly integrable.
Thus, 
$\mathbb{E} U = \lim_{n \to \infty} \mathbb{E}\left(\sum_{i = 1}^n Z_{ni}W_i^2/n\right) = \lim_{n \to \infty} \sum_{i = 1}^n Z_{ni}/n = 0$.
So $U = 0$ almost surely, and hence $\sum_{i = 1}^n Z_{ni}W_i^2/n \to 0$ almost surely.
The term $(b) = \frac{\sigma_1^2 \sum_{i = 1}^n Z_{ni} W_i^2}{\sum_{i = 1}^n a_{ni}^2}\to 0$ since $\sum_{i = 1}^n a_{ni}^2 \asymp n$.

	Combining the above with \eqref{eq:kappaidentity}, we have $\widehat{\tau}^2_n \to \tau^2_0$ almost surely under $P_1$.
	
	Next, we show that $\widehat{\sigma}_n^2\phi_1^{2\nu} \rightarrow \kappa_0$ almost surely under $P_0$.  Since $\widehat{\tau}^2_n \rightarrow \tau^2_0$ almost surely under $P_0$ and $\sigma_1^2 = \kappa_0 / \phi_1^{2\nu}$, it suffices to show that $\widehat{\sigma}^{'2}_n : = \argmin_{\sigma^2 \in [c,d]} \ell(\tau^2_0, \sigma^2, \phi_1)$ converges almost surely to $\sigma_1^2$ under $P_0$. Again, since $P_0 \equiv P_1$, it suffices to show $\widehat{\sigma}^{'2}_n \rightarrow \sigma_1^2$ almost surely under $P_1$. Under $P_1$,
	\begin{equation}
	\label{eq:ellsigma}
	\ell(\tau^2_0, \sigma^2, \phi_1) =\sum_{i = 1}^n \frac{\tau^2_0 + \sigma_1^2 \lambda_i^{(n)} }{\tau^2_0 + \sigma^2 \lambda^{(n)}_i} W_i^2  + \sum_{i = 1}^n \log(\tau^2_0 + \sigma^2 \lambda^{(n)}_i).
	\end{equation}
	Taking the derivative of \eqref{eq:ellsigma} with respect to $\sigma^2$ and equating to zero, we obtain
	\begin{equation}
	\label{eq:sigmaidentity}
	\sum_{i = 1}^n b_{ni} (W_i^2 - 1) = (\widehat{\sigma}^{'2}_n - \sigma_1^2)  \sum_{i = 1}^n b_{ni}^2 W_i^2.
	\end{equation}
	with $b_{ni} = \lambda_i^{(n)} / (\tau^2_0 + \widehat{\sigma}^{'2}_n \lambda_i^{(n)})$. It suffices to prove that 
	$\sum_{i = 1}^n b_{ni} (W_i^2 - 1)/ \sum_{i = 1}^n b_{ni}^2 W_i^2$ converges almost surely to $0$.
	Since
	\begin{equation*}
	\frac{ \sum_{i = 1}^n b_{ni} (W_i^2 - 1)}{\sum_{i = 1}^n b_{ni}^2 W_i^2}
	= \frac{ \sum_{i = 1}^n b_{ni} (W_i^2 - 1)}{\sum_{i = 1}^n b_{ni} } \cdot 
	\frac{ \sum_{i = 1}^n b_{ni}}{\sum_{i = 1}^n b_{ni}^2} \cdot
	\frac{\sum_{i = 1}^n b_{ni}^2}{\sum_{i = 1}^n b_{ni}^2 W_i^2},
	\end{equation*}
	and 
	$\sum_{i = 1}^n b_{ni} \asymp n^{1 /( 2 \nu/d +1)}$,
	$\sum_{i = 1}^n b_{ni}^2 \asymp n^{1/ (2 \nu/d +1)}$ by Lemma \ref{lemma: ab_asymp_B} (2),
	we get
	\begin{equation*}
	\frac{ \sum_{i = 1}^n b_{ni} (W_i^2 - 1)}{\sum_{i = 1}^n b_{ni} }  \longrightarrow 0 \quad \mbox{and} \quad \frac{\sum_{i = 1}^n b_{ni}^2}{\sum_{i = 1}^n b_{ni}^2 W_i^2} \longrightarrow 1 \quad a.s.
	\end{equation*}
	Combining the above estimates with \eqref{eq:sigmaidentity}, we have $\widehat{\sigma}^{'2}_n \to \sigma_1^2$ almost surely under $P_1$.
\end{proof}

It is difficult to establish the consistency of the joint maximum likelihood estimates of $\{\kappa,\tau^2,\phi\}$ (i.e., $\phi$ is not fixed). A related result can be found in Theorem~2 of \cite{KS13} without a nugget effect. 
In the presence of a nugget effect, constructing such a proof becomes difficult due to the analytic intractability of the maximum likelihood estimators for $\{\kappa,\tau^2,\phi\}$. Nevertheless, our simulation studies in Section~\ref{subsec: par_estimate} seem to support consistent estimation of $\{\kappa,\tau^2\}$ even when $\phi$ is not fixed.

\subsubsection{Asymptotic normality of the maxumum likelihood estimator}
Given the consistency of the maximum likelihood estimators, we turn to their asymptotic distributions.  
For simplicity of presentation, we let $S = [0,1]^d$ in the following theorem.
The asymptotic normality described below holds for any compact set $S \subset \mathbb{R}^d$. 
\begin{theorem}
	\label{thm:CLT}
	Assume that $n$ is the $d^{th}$ power of some positive integer, $ \chi_{n} =  [0,1)^d \cap n^{-1/d} \mathbb{Z}^d$, and the conditions in Assumption~\ref{assump:2} hold.
	Let 
	\begin{equation*}
	a^0_{ni}: = 1/(\tau^2_0 + \sigma_1^2 \lambda_i^{(n)}) \quad \mbox{and} \quad b^0_{ni}: = \lambda_i^{(n)} a_{ni}^0 \quad \mbox{for } \, 1 \le i \le n.
	\end{equation*}
	There exist constants $c_1, c_2, c_3 > 0$ such that as $n \rightarrow \infty$,
	\begin{equation}
	\label{eq:limits}
	\frac{1}{n} \sum_{i = 1}^n (a_{ni}^0 )^2 \to c_1, \quad
	\frac{1}{n} \sum_{i = 1}^n (a_{ni}^0 )^4 \to c_2, \quad 
	\frac{1}{n^{1/(1+ 2 \nu/d)}} \sum_{i = 1}^{n} (b_{ni}^0)^2 \to c_3.
	\end{equation}
	We have 
	\begin{equation}
	\label{eq:CLTkappa}
	\sqrt{n}(\widehat{\tau}^2_n - \tau^2_0) \stackrel{(d)}{\longrightarrow} \mathcal{N}(0, 2 \tau_0^4 c_2/c_1^2),
	\end{equation}
	and
	\begin{equation}
	\label{eq:CLTsigma}
	n^{1/(2+4 \nu/d)}(\widehat{\sigma}_n^2\phi_1^{2\nu} - \kappa_0) \stackrel{(d)}{\longrightarrow} \mathcal{N}(0, 2\phi_1^{4\nu}/c_3),
	\end{equation}
under $P_1$ corresponding to the Mat\'ern model with covariogram 
$K(\cdot; \, \tau_0^2, \sigma_1^2, \phi_1, \nu)$, with $\sigma_1^{2}: = \kappa_0 / \phi_1^{2 \nu}$.
\end{theorem}
\begin{proof}
	With Assumption \ref{assump:2}, the limits in \eqref{eq:limits} follow from Lemma \ref{lemma: ab_asymp_B} (3). By \eqref{eq:kappaidentity} and Theorem~\ref{thm:consistency}, we have
	\begin{eqnarray} \label{eq:kappaThm2}
	\sqrt{n}(\tau^2_0 - \widehat{\tau}^2_n) = 
	(1 + o(1)) \, \frac{\tau^2_0 \sqrt{n} \sum_{i = 1}^n (1-W_i^2) (a_{ni}^0)^2+ \sigma_1^2 \sqrt{n}  \sum_{i = 1}^n (1-W_i^2) \lambda^n_i (a^0_{ni})^2}{\sum_{i=1}^n W_i^2 (a_{ni}^0)^2}.
	\end{eqnarray}
	We know that $\sum_{i=1}^n W_i^2 (a_{ni}^0)^2/\sum_{i = 1}^n (a_{ni}^0 )^2 \longrightarrow 1$.  
	In addition,
	\begin{align}
	\label{eq:mainterm}
	\frac{\tau^2_0 \sqrt{n} \sum_{i = 1}^n (1-W_i^2) (a_{ni}^0)^2}{\sum_{i = 1}^n (a_{ni}^0 )^2} = \frac{ \sum_{i = 1}^n (1-W_i^2) (a_{ni}^0)^2}{\sqrt{2 \sum_{i = 1}^n (a_{ni}^0)^4}} \cdot \frac{\tau^2_0 \sqrt{2n \sum_{i = 1}^n (a_{ni}^0)^4}}{\sum_{i = 1}^n (a_{ni}^0 )^2}  
	\stackrel{(d)}{\longrightarrow} \mathcal{N}(0, 2 \tau_0^4 c_2/c_1^2),
	\end{align}
	where the first term on the right hand side converges to $\mathcal{N}(0,1)$ by Lindeberg's central limit theorem, and the second term converges to $\tau^2_0 \sqrt{2c_2}/c_1$.
	Similarly,
	\begin{equation}
	\label{eq:smallterm}
	\frac{\sigma_1^2 \sqrt{n}  \sum_{i = 1}^n (1-W_i^2) \lambda^n_i (a^0_{ni})^2}{\sum_{i = 1}^n (a_{ni}^0 )^2} = 
	\frac{\sum_{i = 1}^n (1-W_i^2) \lambda^n_i (a^0_{ni})^2}{\sqrt{2 \sum_{i = 1}^n (\lambda_i^{n})^2 (a^0_{ni})^4}} \cdot \frac{\sigma_1^2 \sqrt{2 n \sum_{i = 1}^n (\lambda_i^{n})^2 (a^0_{ni})^4}}{\sum_{i = 1}^n (a_{ni}^0 )^2} \longrightarrow 0,
	\end{equation}
	where the first term on the right hand side converges to $\mathcal{N}(0,1)$, and the second term converges to $0$ since $\sum_{i = 1}^n (\lambda_i^{n})^2 (a^0_{ni})^4 \asymp n^{1/(1+ 2 \nu/d)}$.
	Combining \eqref{eq:kappaThm2}, \eqref{eq:mainterm} and \eqref{eq:smallterm} leads to \eqref{eq:CLTkappa}.
	
	By \eqref{eq:sigmaidentity} and Theorem~\ref{thm:consistency}, we get
	\begin{equation}
	\label{eq:sigmaThm2}
	n^{1/(2+4 \nu/d)}(\widehat{\sigma}_n^2 - \sigma_1^2)  =  (1 + o(1)) \, \frac{n^{1/(2+4 \nu/d)} \sum_{i = 1}^n b_{ni}^0 (W_i^2 -1)}{\sum_{i=1}^n (b_{ni}^0)^2 W_i^2}.
	\end{equation}
	Moreover,
	\begin{align}
	\label{eq:mainterm2}
	\frac{n^{1/(2+4 \nu/d)} \sum_{i = 1}^n b_{ni}^0 (W_i^2 -1)}{\sum_{i=1}^n (b_{ni}^0)^2 W_i^2} & = \frac{\sum_{i = 1}^n b_{ni}^0 (W_i^2 -1)}{\sqrt{2 \sum_{i = 1}^n (b_{ni}^0)^2}} \cdot \frac{\sqrt{2} n^{1/(2+4 \nu/d)}}{\sqrt{\sum_{i = 1}^{n} (b_{ni}^0)^2}} \cdot \frac{\sum_{i = 1}^{n} (b_{ni}^0)^2}{\sum_{i=1}^n (b_{ni}^0)^2 W_i^2} \notag \\
	& \stackrel{(d)}{\longrightarrow} \mathcal{N}(0,2/c_3),
	\end{align}
	where the first term on the right hand side converges to $\mathcal{N}(0,1)$, the second term converges to $\sqrt{2/c_3}$, and the third term converges to $1$. Combining \eqref{eq:sigmaThm2} and \eqref{eq:mainterm2} yields \eqref{eq:CLTsigma}.
\end{proof}

\cite{DZM09} showed that for the Mat\'ern model without measurement error, the maximum likelihood estimator $\widehat{\sigma}_n^2$ converges to $\sigma_1^2$ at a $\sqrt{n}$-rate. Theorem \ref{thm:CLT} shows that in the presence of measurement error, the maximum likelihood estimator $\widehat{\tau}^2_n$ has a $\sqrt{n}$-rate while $\widehat{\sigma}_n^2$ has a slower $n^{1/(2+4 \nu/d)}$-rate.
This echoes the results of \cite{Ying91}, \cite{CSY00} with $\nu = \frac{1}{2}$ and $d = 1$ for the Ornstein-Uhlenbeck process, where the maximum likelihood estimator $\widehat{\sigma}_n^2$ converges at a $\sqrt{n}$-rate without measurement error, but at a $\sqrt[4]{n}$-rate in the presence of measurement error.


\subsection{Interpolation at new locations}\label{subsec: prediction}
We now turn to predicting the value of the process at unobserved locations. Without the nugget (i.e., $\tau = 0$ in (\ref{eq:basic})), \cite{Stein88, Stein93, Stein99} establish that 
predictions under different measures tend to agree as sample size $n \rightarrow \infty$.
However, in the presence of a nugget effect, 
the predictive variance of $y(s)$ at an unobserved location may not decrease to zero with increasing sample size.
In fact, the squared prediction error for any linear predictor is expected to be at least $\tau^2$. For example, let $\widehat{y}_0 = v^\top y$ be a linear predictor of $y_0 = y(s_0)$ at the unobserved location $s_0, s_0 \not\in \chi_n$. Let $w = \{w(s_1), \ldots, w(s_n)\}$, $\epsilon = \{\epsilon(s_1), \ldots, \epsilon(s_n)\}$, $w_0 = w(s_0)$ and $\epsilon_0 = \epsilon(s_0)$. The expected squared prediction error satisfies  
\begin{align*}
\mathbb{E}[(\widehat{y}_0 - y_0)^2] & = \mathbb{E}[\{(v^\top w - w_0) + (v^\top \epsilon - \epsilon_0)\}^2] = \mathbb{E}[(v^\top w - w_0)^2] +  \mathbb{E}[(v^\top \epsilon - \epsilon_0)^2]  \ge \tau^2.
\end{align*}
To see whether there can be a consistent linear (unbiased) estimate of the underlying process $w(\cdot)$ at unobserved locations, consider the universal kriging estimator at 
an unobserved location $s_0$ given by
\begin{equation}
\label{eq:BLUE}
\widehat{Z}_n(\tau^2, \sigma^2, \phi) : = \gamma_n(\sigma^2, \phi)^\top \Gamma_n(\tau^2, \sigma^2, \phi)^{-1} y\;,
\end{equation}
where $\{\gamma_n(\sigma^2, \phi)\}_i: = K_w(s_0 - s_i; \, \sigma^2, \phi, \nu)$, and $\{\Gamma_n(\tau^2, \sigma^2, \phi)\}_{ij}: = K_w(s_i - s_j; \, \sigma^2, \phi, \nu) + \tau^2 \delta_0(i-j)$ for $i, j = 1, \ldots, n$. The interpolant $\widehat{Z}_n(\tau^2, \sigma^2, \phi)$ provides a best linear unbiased estimate of $w_0$ under the Mat\'ern model with measurement error \eqref{eq:Maternmeas}.
By letting $\{K_n(\phi)\}_{ij} : = K_w(s_0 - s_i; \, 1, \phi, \nu)$, we have the mean squared error of the estimator \eqref{eq:BLUE}  follows
\begin{multline}
\label{eq:varpred}
\var_{\tau_0^2, \sigma_0^2, \phi_0}\{\widehat{Z}_n(\tau^2, \sigma^2, \phi) - w_0\} 
= \sigma_0^2 \{ 1 - 2 \gamma_n(\sigma^2, \phi)^\top \Gamma_n(\tau^2, \sigma^2, \phi)^{-1} \gamma_n(\sigma_0^2, \phi_0)  \\
+  \gamma_n(\sigma^2, \phi)^\top \Gamma_n(\tau^2, \sigma^2, \phi)^{-1} K_n(\phi_0) \Gamma_n(\tau^2, \sigma^2, \phi)^{-1}  \gamma_n(\sigma^2, \phi)  \}  \\
+ \tau_0^2\gamma_n(\sigma^2, \phi)^\top \Gamma_n(\tau^2, \sigma^2, \phi)^{-2} \gamma_n(\sigma^2, \phi)\; , 
\end{multline}
where $\{\tau_0^2, \sigma_0^2, \phi_0\}$ are the true generating values of $\{\sigma^2, \phi, \tau^2\}$. 
Setting $(\tau^2, \sigma^2, \phi) = (\tau_0^2, \sigma_0^2, \phi_0)$ in \eqref{eq:varpred} yields
\begin{eqnarray}
\label{eq:varpred2}
\var_{\tau_0^2, \sigma_0^2, \phi_0}\{\widehat{Z}_n(\tau_0^2, \sigma_0^2, \phi_0) - w_0\}  = 
\sigma_0^2 \{ 1 - \gamma_n(\sigma_0^2, \phi_0)^\top \Gamma_n(\tau_0^2, \sigma_0^2, \phi_0)^{-1} \gamma_n(\sigma_0^2, \phi_0)\}
\end{eqnarray}
Theorem~8 in Chapter~3 of \cite{Stein99} characterizes the mean squared error of the best linear unbiased estimate at location $0$ as $\displaystyle \frac{(2 \pi c)^{1/\alpha}}{\alpha \sin \left( {\pi/\alpha}\right)} \left( \delta \tau^2 \right)^{1 - 1/\alpha}$
with observations at $\delta j$ for $j \neq 0$. Here $\alpha: = 2 \nu + 1$ and $c: = C \sigma^2 \phi^{2 \nu}$ with $C$ defined in \eqref{eq:spectralb}.
Following the same argument, it is not hard to see that the mean squared error of the best linear unbiased estimate (based on data in $\mathbb{R}^d$) is of order $\delta^{2 \nu / (2 \nu + d)}$. \cite{Stein99} proved this for observations on the whole line (with a typo in the expression $(44)$ of \cite{Stein99}). He also conjectured that the 
above expression for the mean-square error holds for data on any finite interval. We conduct simulations in Section~\ref{subsec: interpolation} with the nugget effect to corroborate this. 

{\subsection{Covariance tapering}\label{subsec: tapering}
Covariance tapering \citep{fur06, KSN08, DZM09} approximates the likelihood by setting certain entries of the covariance matrix to zero to introduce sparsity and, hence, achieve computational benefits. In the presence of a nugget, we explore parameter estimation for the Mat\'ern model \eqref{eq:Maternmeas} with covariance tapering, which, too,  have been investigated without the nugget by \cite{wang2011fixed}. Let $K_{\tiny taper}(x; \gamma)$ be a tapering function, which is an isotropic correlation function such that $K_{\tiny taper}(x; \gamma) = 0$ for $|x| > \gamma$.
The tapered covariogram of the Mat\'ern model with measurement error is given by 
\begin{equation} \label{eq:taperMatern}
\widetilde{K}(x; \tau^2, \sigma^2, \phi, \nu, \gamma) = K(x;\tau^2, \sigma^2, \phi, \nu)\, K_{\tiny taper}(x; \gamma)\;,
\end{equation}
where $K(x;\tau^2, \sigma^2, \phi, \nu)$ is defined in \eqref{eq:Maternmeas}.
Recalling the notations from Section~\ref{subsec: estimation}, we obtain the tapered covariance matrix of the observations $y = (y_1, \ldots, y_n)^\top$ as 
\begin{equation}
\widetilde{V}_n = V_n \circ T(\gamma) = \tau^2 I_n + K_n \circ T(\gamma),
\end{equation}
and the (rescaled) negative log-likelihood is $\widetilde{\ell}(\tau, \sigma^2, \phi): = \log \det \widetilde{V}_n + y^{\top} \widetilde{V}_n^{-1} y$, 
where 
$T(\gamma)$ is the $n \times n$ matrix with $(i,j)$-th entry $K_{\tiny taper}(s_i - s_j; \gamma)$ and $\circ$ denotes the element-wise (Schur or Hadamard) matrix product. For any fixed $\phi_1 > 0$, let $(\widehat{\tau}^2_{{\tiny taper}, n}(\phi_1), \widehat{\sigma}^2_{{\tiny taper}, n}(\phi_1))$ be the maximum likelihood estimators of the tapered Mat\'ern model, i.e.,
\begin{equation}\label{eq:MLEtaper}
(\widehat{\tau}^2_{{\tiny taper}, n}(\phi_1), \widehat{\sigma}^2_{{\tiny taper}, n}(\phi_1))= \argmin_{(\tau^2, \sigma^2) \in D} \widetilde{\ell}(\tau^2, \sigma^2, \phi_1)\;.
\end{equation}
To address the identifiability issue of the tapered Mat\'ern model, we require the following assumption on the tapering function which is due to \cite{KSN08}.
\begin{assumption} \label{assump:3}
The spectral density $f_{\tiny taper}(u)$ of the tapering function $K_{\tiny taper}(\cdot; \gamma)$ exists, and that
there exist $\varepsilon > \max \{\frac{d}{4}, 1 - \nu \}$ and $M_{\varepsilon} < \infty$ such that
\begin{equation} \label{eq:densityrestriction}
f_{\tiny taper}(u) \le \frac{M_{\varepsilon}}{(1+u^2)^{\nu + \frac{d}{2} + \varepsilon}}, \quad u \ge 0.
\end{equation}
\end{assumption}

\begin{theorem}
\label{thm:equivtapering}
For $d \le 3$, let $S \subset \mathbb{R}^d$ be a compact set. 
For $i = 1,2$, let $\widetilde{P}_i$ be the probability measure of the Gaussian process on $S$ with mean zero and covariance $\widetilde{K}(\cdot; \tau_i^2, \sigma_i^2, \phi_i, \nu, \gamma)$ defined by \eqref{eq:taperMatern}. Under the conditions in Assumption~\ref{assump:3}, we have the following results:
(i) if $\tau_1^2 \ne \tau_2^2$, then $\widetilde{P}_1 \perp \widetilde{P}_2$; and (ii) if $\tau_1^2 = \tau_2^2$, then $\widetilde{P}_1 \equiv \widetilde{P}_2$ if and only if $\sigma_1^2 \phi_1^{2 \nu} = \sigma_2^2 \phi_2^{2 \nu}$.
\end{theorem}
\begin{proof}
We know \cite[Theorem~1]{KSN08} that $\widetilde{P}_i \equiv P_i$ for $i = 1,2$ under Assumption~\ref{assump:3}. Therefore, the proof is an immediate consequence of our Theorem~\ref{thm:identify} in Section~\ref{sec: asymp_thm}.
\end{proof}

To progress further, we recall the crucial role of the eigenvalues of $\frac{1}{\sigma^2}K_n$ in analyzing the maximum likelihood estimators of the Mat\'ern covariogram parameters with measurement error. With covariance tapering, we need estimates on the eigenvalues of $\frac{1}{\sigma^2}K_n \circ T(\gamma)$. Let $\{\widetilde{\lambda}_{i}^{(n)}, i = 1, \ldots, n\}$ be the eigenvalues of $\frac{1}{\sigma^2}K_n \circ T(\gamma)$ in decreasing order. Under Assumption \ref{assump:3}, the spectral density $\widetilde{f}$ of the tapered Mat\'ern model with covariogram \eqref{eq:taperMatern} satisfies $\widetilde{f}(u) \asymp f(u) \asymp u^{-2 \nu -d}$ ((B.1) in \cite{KSN08}). By applying Theorem \ref{thm:BT2}, we have for $\max_{s \in S} \min_{1 \le i \le n} \|s - s_i\| \asymp n^{-1/d}$, 
\begin{equation}
\widetilde{\lambda}_{i}^{(n)} \le Cn i^{-2 \nu/d -1} \quad \mbox{for all } i = 1, \ldots, n.
\end{equation}
In order to further study the maximum likelihood estimates of the tapered Mat\'ern model, we need some assumptions on the eigenvalues $\{\widetilde{\lambda}_{i}^{(n)}, i = 1, \ldots, n\}$. The following two assumptions are analogues of Assumptions~\ref{assump:1}~and~\ref{assump:2}.

\begin{assumption}\label{assump:4}
	Assume that $\min_{1 \le i \ne j \le n} ||s_i - s_j|| \asymp n^{-1/d}$. 
	There exists $c >0$ such that 
	\begin{equation}\label{lambda_lower_B_taper}
	\widetilde{\lambda}_i^{(n)} \geq c n i^{-2\nu/d -1} \quad  \mbox{ for all } \,  i = 1, \ldots, n.
	\end{equation}
\end{assumption}

\begin{assumption}\label{assump:5}
	Let $\chi_n= [0,1)^d \cap n^{-1/d} \mathbb{Z}^d$ be the regular grid.
	There exists $A = A(\phi, \nu, d)> 0$ such that
	\begin{equation}\label{lambda_lim_taper}
	\widetilde{\lambda}_i^{(n)} / (n i^{-2\nu/d-1}) \to A \quad \mbox{ as } \, n, i \to \infty
	\end{equation}
\end{assumption}
In Figure~\ref{fig:assump4} we plot the values of $\widetilde{\lambda}_i^{(n)} / (n i^{-2\nu-1})$ with sampled points on the regular grid $[0,1) \cap n^{-1} \mathbb{Z}$ for $\nu = 0.9, 1.5$, $n$ ranging from $500$ to $4000$, and $i = n^{0.7}$, $n^{0.8}$, $n^{0.9}$. The tapering function for obtaining $\widetilde{\lambda}_i^{(n)}$ is a stationary Wendland function $K_{taper}(x; \gamma) = (1 - |x|/\gamma)_{+}^4(1 + 4|x|/\gamma)$ where $\gamma = 0.5$ \citep{wendland95}. Consistent with Assumption~\ref{assump:4}~\&~\ref{assump:5}, the profile plots of $\widetilde{\lambda}_i^{(n)} / (n i^{-2\nu/d-1})$ flatten as $n$ increases and the quantity $\widetilde{\lambda}_i^{(n)} / (n i^{-2\nu-1})$ tends to converge as $n, i$ become large. 
\begin{figure}[h]
	\centering
	\includegraphics[width = 0.9\linewidth]{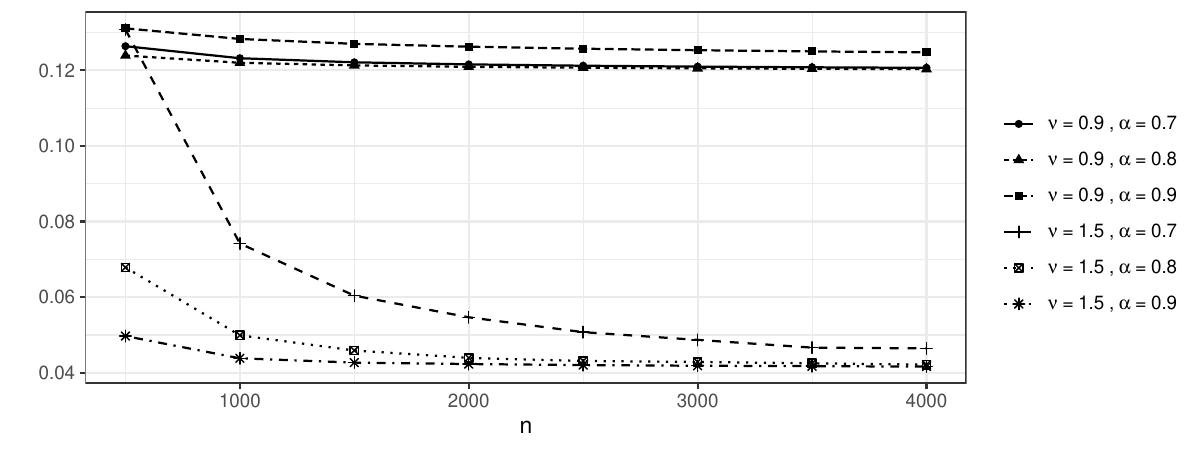}
	\caption{{Trend of $\widetilde{\lambda}_i^{(n)} / (n i^{-2\nu/d-1})$ for $i = n^\alpha$ when the points are sampled on the regular grid $[0,1) \cap n^{-1} \mathbb{Z}$. Parameters $\phi$ and $\sigma^2$ in Mat\'ern covariogram are set to be $1.0$ and $1.0$, respectively.}
		\label{fig:assump4}}
\end{figure}
}

{
Now we state the consistency results for the maximum likelihood estimators of the tapered Mat\'ern model.

\begin{theorem}\label{thm:consistencytaper}
	Assume that $(\tau_0^2, \sigma_0^2) \in D$, $\chi_n: = \{s_1, \ldots, s_n\}$ satisfy
	\[
	\max_{s \in S} \min_{1 \le i \le n} \|s - s_i\| \asymp n^{-1/d} \quad \mbox{and} \quad \min_{1 \le i \ne j \le n} ||s_i - s_j|| \asymp n^{-1/d},
	\]
and the conditions in Assumption~\ref{assump:3} hold. Let $\widetilde{P}_0$ be the probability measure of the tapered Mat\'ern model with covariogram $\widetilde{K}(\cdot; \tau_0^2, \sigma_0^2, \phi_0, \nu, \gamma)$. 
\begin{enumerate}
\item[(1)]
We have $\widehat{\tau}_{{\tiny taper}, n}^2 \rightarrow \tau_0^2$ almost surely under $\widetilde{P}_0$.
\item[(2)]
Assume that the conditions in Assumption~\ref{assump:4} hold.
Then $\widehat{\sigma}_{{\tiny taper}, n}^2 \phi_1^{2 \nu} \rightarrow \kappa_0$ almost surely under $\widetilde{P}_0$
\item[(3)]
Assume that $n$ is the $d^{th}$ power of some positive integer, $ \chi_{n} =  [0,1)^d \cap n^{-1/d} \mathbb{Z}^d$, and the conditions in Assumptions~\ref{assump:3} and \ref{assump:5} hold. Let $\widetilde{a}^0_{ni}: = 1/(\tau^2_0 + \sigma_1^2 \widetilde{\lambda}_i^{(n)})$ and $\widetilde{b}^0_{ni}: = \widetilde{\lambda}_i^{(n)} a_{ni}^0$ for $\, 1 \le i \le n$.
Then, there exist constants $\widetilde{c}_1, \widetilde{c}_2, \widetilde{c}_3 > 0$ such that as $n \rightarrow \infty$,
	\begin{equation}
	\label{eq:limits2}
	\frac{1}{n} \sum_{i = 1}^n (\widetilde{a}^0_{ni})^2 \to \widetilde{c}_1, \quad
	\frac{1}{n} \sum_{i = 1}^n (\widetilde{a}^0_{ni} )^4 \to \widetilde{c}_2, \quad 
	\frac{1}{n^{1/(1+ 2 \nu/d)}} \sum_{i = 1}^{n} (\widetilde{b}^0_{ni})^2 \to \widetilde{c}_3.
	\end{equation}
	We also have
	\begin{equation}
	\label{eq:CLTkappa2}
	\sqrt{n}(\widehat{\tau}^2_{{\tiny taper}, n} - \tau^2_0) \stackrel{(d)}{\longrightarrow} \mathcal{N}(0, 2 \tau_0^4 \widetilde{c}_2/\widetilde{c}_1^2),
	\end{equation}
	and
	\begin{equation}
	\label{eq:CLTsigma2}
	n^{1/(2+4 \nu/d)}(\widehat{\sigma}_{{\tiny taper}, n}^2\phi_1^{2\nu} - \kappa_0) \stackrel{(d)}{\longrightarrow} \mathcal{N}(0, 2\phi_1^{4\nu}/\widetilde{c}_3).
	\end{equation}
	under $\widetilde{P}_1$ corresponding to the tapered Mat\'ern model with covariogram $\widetilde{K}(\cdot; \, \tau_0^2, \sigma_1^2, \phi_1, \nu, \gamma)$, with $\sigma_1^{2}: = \kappa_0 / \phi_1^{2 \nu}$.
\end{enumerate}
\end{theorem}
\begin{proof}
From Theorem \ref{thm:equivtapering}, we know that $\widetilde{P}_0 \equiv \widetilde{P}_1$. Under $\widetilde{P}_1$, the (rescaled) negative log-likelihood is written as 
\begin{equation}
	\ell(\tau^2, \widehat{\tau}^2_{{\tiny taper}, n}, \phi_1) = \sum_{i = 1}^n \frac{\tau^2_0 + \sigma_1^2 \widetilde{\lambda}_i^{(n)}}{\tau^2 + \widehat{\sigma}^2_{{\tiny taper}, n}\widetilde{\lambda}_i^{(n)}} W_i^2 + \sum_{i=1}^n \log(\tau^2 + \widehat{\sigma}^2_{{\tiny taper}, n}\widetilde{\lambda}_i^{(n)})\;,
\end{equation}
where $W_i \stackrel{iid}{\sim} \mathcal{N}(0,1)$.
The remainder of the proof follows analogously to Theorems~\ref{thm:consistency}~and~\ref{thm:CLT} 
by using Assumptions~\ref{assump:4}~and~\ref{assump:5} instead of Assumptions~\ref{assump:1}~and~\ref{assump:2}.
\end{proof}
}

{
\subsection{Consistency and asymptotic normality for $d \ge 5$} \label{subsec: highd}

In contrast to $d \le 3$, the parameters $\{\tau^2, \sigma^2, \phi\}$ are consistently estimable for $d \ge 5$. It is, therefore, of interest to establish if the maximum likelihood estimators of $\{\tau^2, \sigma^2, \phi\}$ are consistent in $d\ge 5$. Here we consider a slightly weaker version of the problem which should offer sufficient insights into methods for Gaussian processes for $d \ge 5$.

Recall the development in Section~\ref{subsec: estimation}.
Since the scale parameter $\phi_0$ is consistently estimable, there exists an estimator $\widehat{\phi}'_n$ such that 
$\widehat{\phi}'_n \to \phi_0$ almost surely ($\widehat{\phi}'_n$ can be any consistent estimator of $\phi_0$).  
Let $(\widehat{\tau}^2(\widehat{\phi}'_n), \widehat{\sigma}^2_{n}(\widehat{\phi}'_n))$ be the maximum likelihood estimators based on the estimator $\widehat{\phi}'_n$:
\begin{equation}\label{eq:MLEhighd}
(\widehat{\tau}^2(\widehat{\phi}'_n), \widehat{\sigma}^2_{n}(\widehat{\phi}'_n))= \argmin_{(\tau^2, \sigma^2) \in D} \ell(\tau^2, \sigma^2, \widehat{\phi}'_n).
\end{equation}
The next theorem establishes consistency of the maximum likelihood estimators $(\widehat{\tau}^2(\widehat{\phi}'_n), \widehat{\sigma}^2_{n}(\widehat{\phi}'_n))$.

\begin{theorem}\label{thm:consistencyhighd}
Assume that $(\tau_0^2, \sigma_0^2) \in D$ and the locations in $\chi_n: = \{s_1, \ldots, s_n\}$ satisfy
	\[
	\max_{s \in S} \min_{1 \le i \le n} \|s - s_i\| \asymp n^{-1/d} \quad \mbox{and} \quad \min_{1 \le i \ne j \le n} ||s_i - s_j|| \asymp n^{-1/d}.
        \]
Let $P_0$ be the probability measure of the tapered Mat\'ern model with covariogram $K(\cdot; \tau_0^2, \sigma_0^2, \phi_0, \nu)$. 
\begin{enumerate}
\item[(1)]
We have $\widehat{\tau}^2(\widehat{\phi}'_n) \rightarrow \tau_0^2$ almost surely under $P_0$.
\item[(2)]
Under Assumption~\ref{assump:1},
	$\widehat{\sigma}^2_{n}(\widehat{\phi}'_n) \rightarrow \sigma_0^2$ almost surely under $P_0$.
\item[(3)] Let $n$ be the $d^{th}$ power of some positive integer, $ \chi_{n} =  [0,1)^d \cap n^{-1/d} \mathbb{Z}^d$, and suppose Assumption~\ref{assump:2} holds. Let $\overline{a}^{0}_{ni}: = 1/(\tau^2_0 + \sigma_0^2 \lambda_i^{(n)})$ and $\overline{b}^0_{ni}: = \lambda_i^{(n)} \overline{a}_{ni}^0 $ for $1 \le i \le n$. 
	Then, there exist constants $\overline{c}_1, \overline{c}_2, \overline{c}_3 > 0$ such that as $n \rightarrow \infty$,
	\begin{equation}
	\label{eq:limits3}
	\frac{1}{n} \sum_{i = 1}^n (\overline{a}_{ni}^0 )^2 \to \overline{c}_1, \quad
	\frac{1}{n} \sum_{i = 1}^n (\overline{a}_{ni}^0 )^4 \to \overline{c}_2, \quad 
	\frac{1}{n^{1/(1+ 2 \nu/d)}} \sum_{i = 1}^{n} (\overline{b}_{ni}^0)^2 \to \overline{c}_3.
	\end{equation}
	We also have
	\begin{equation}
	\label{eq:CLTkappa3}
	\sqrt{n}(\widehat{\tau}^2(\widehat{\phi}'_n)- \tau^2_0) \stackrel{(d)}{\longrightarrow} \mathcal{N}(0, 2 \tau_0^4 \overline{c}_2/\overline{c}_1^2),
	\end{equation}
	and
	\begin{equation}
	\label{eq:CLTsigma3}
	n^{1/(2+4 \nu/d)}(\widehat{\sigma}^2_{n}(\widehat{\phi}'_n) - \sigma_0^2) \stackrel{(d)}{\longrightarrow} \mathcal{N}(0, 2/\overline{c}_3).
	\end{equation}
\end{enumerate}
\end{theorem}
\begin{proof}
To study the asymptotic properties of $(\widehat{\tau}^2(\widehat{\phi}'_n), \widehat{\sigma}^2_{n}(\widehat{\phi}'_n))$, it suffices to consider
$(\widehat{\tau}^2(\phi_0), \widehat{\sigma}^2_{n}(\phi_0)) = \argmin_{(\tau^2, \sigma^2) \in D} \ell(\tau^2, \sigma^2,\phi_0)$. Recalling that $P_0$ is the probability measure of the Mat\'ern model with covariogram $K(\cdot; \tau_0^2, \sigma_0^2, \phi_0, \nu)$, the (rescaled) negative log-likelihood \eqref{eq: neg_loglik} is written as 
\begin{equation}
	\ell(\tau^2, \sigma^2, \phi_0) = \sum_{i = 1}^n \frac{\tau^2_0 + \sigma_0^2 \lambda_i^{(n)}}{\tau^2 + \sigma^2 \lambda_i^{(n)}} W_i^2 + \sum_{i=1}^n \log(\tau^2 + \sigma^2 \lambda_i^{(n)})
\end{equation}
under $P_0$, where $W_i \stackrel{iid}{\sim} \mathcal{N}(0,1)$. The reasoning in Theorems~\ref{thm:consistency}~and~\ref{thm:CLT} shows that $(\widehat{\tau}^2(\phi_0), \widehat{\sigma}^2_{n}(\phi_0))$ are consistent and are asymptotically normal under various assumptions. 
As a result, the same holds for $(\widehat{\tau}^2(\widehat{\phi}'_n), \widehat{\sigma}^2_{n}(\widehat{\phi}'_n))$. This completes the proof.
\end{proof}
}


\section{Simulations}\label{sec: simulation}
\subsection{Set-up}\label{subsec: set-up}
The preceding results help explain the behaviour of the inference from (\ref{eq:basic}) as the sample size increases within a fixed domain. Here, we present some simulation experiments to illustrate statistical inference for finite samples. 
We simulate data sets based on \eqref{eq:basic} in a unit square setting $\nu = 1/2$ and $\sigma^2 = 1$. 
We pick three different values of the nugget, $\tau^2 \in \{0, 0.2, 0.8\}$, and choose the decay parameter $\phi$ so that the effective spatial range is $0.15$, $0.4$ or $1$, i.e., the correlation decays to $0.05$ at a distance of $0.15$, $0.4$ or $1$ units. Therefore, we consider $3 \times 3 = 9$ different parameter settings. For each parameter setting, we simulate $1000$ realizations of the Gaussian process over $n = 1600$ observed locations.
The observed locations are chosen from a perturbed grid. We construct a $67 \times 67$ regular grid with coordinates from 0.005 to 0.995 in increments of 0.015 in each dimension. We add a uniform $[-0.005, 0.005]^2$ perturbation to each grid point to ensure at least 0.005 units separation from its nearest neighbour. We then choose $n = 1600$ locations out of the perturbed grid. 
Codes for studies in this Section are available on \url{https://github.com/LuZhangstat/nugget_consistency}.

\subsection{Likelihood comparisons} 
Theorem~\ref{thm:identify} suggests that it is difficult to distinguish between the two Mat\'ern models with measurement error when their microergodic parameters $\{\kappa,\tau^2\}$ are close to each other. This property should be reflected in the behaviour of the likelihood function for a large finite sample. To see this, we plot interpolated maps of the log-likelihood among different grids of parameter values. We consider the three values of $\tau_0^2$ in Section~\ref{subsec: set-up} and $\phi_0 = 7.49$, which implies an effective spatial range of approximately $0.4$ units, and pick $n=900$ observations from the first realization generated from \eqref{eq:basic}. This yields three different data sets corresponding to the three values of $\tau_0^2$. 
We map the negative one-half of the log-likelihood in \eqref{eq: neg_loglik}. 

The interpolated maps of the log-likelihood are provided in Fig.~\ref{fig:sim1} as a function of $(\tau^2,\phi)$ in the first two rows and of $(\sigma^2,\phi)$ in the third row. The first column presents cases with $\tau_0 = 0$, while the second and the third columns are for $\tau_0 = 0.2$ and $0.8$, respectively. The grid for $\phi$ ranges from $2.5$ to $30$ so that the effective spatial ranges between $0.1$ and $1.2$. We specify the range of $\tau^2$ and $\sigma^2$ to be $(0.0, 1.0)$ and $(0.2, 4.2)$, respectively, so that the pattern of the log-likelihood map around the true generating values of parameters can be captured. All the interpolated maps, including the contour lines, are drawn to the same scale. 

The first row of Figure~\ref{fig:sim1} corresponds to $\sigma^2 = \sigma^2_0=1$, the second row corresponds to $\kappa = \kappa_0$ and the third row corresponds to $\tau^2 = \tau_0^2$. In the first row, we observe that similar log-likelihoods are located along parallel lines $\phi + \tau^2 = Const$. This suggests that one can identify the maximum with either a fixed $\phi$ or $\tau^2$ when $\sigma^2 = \sigma^2_0$. 
In the second row, we find that contours for high log-likelihood values are situated around the actual generating value of the nugget, supporting the identifiability of the nugget as provided in Theorem~\ref{thm:identify}. The log-likelihood along the $\phi$-axis has a flat tail as $\phi$ decreases when fixing the nugget, which indicates having the same value of the microergodic parameter $\kappa = \sigma^2\phi^{2\nu}$ can result in equivalent probability measures (Theorem~\ref{thm:identify}). 
Finally, the third row reveals that the log-likelihood closely follows the curve $\sigma^2\phi = Const$, thereby corroborating Theorem~\ref{thm:identify}. 
\begin{figure}[t]
	\vspace*{-30pt}
	\centering
	\includegraphics[width = \linewidth]{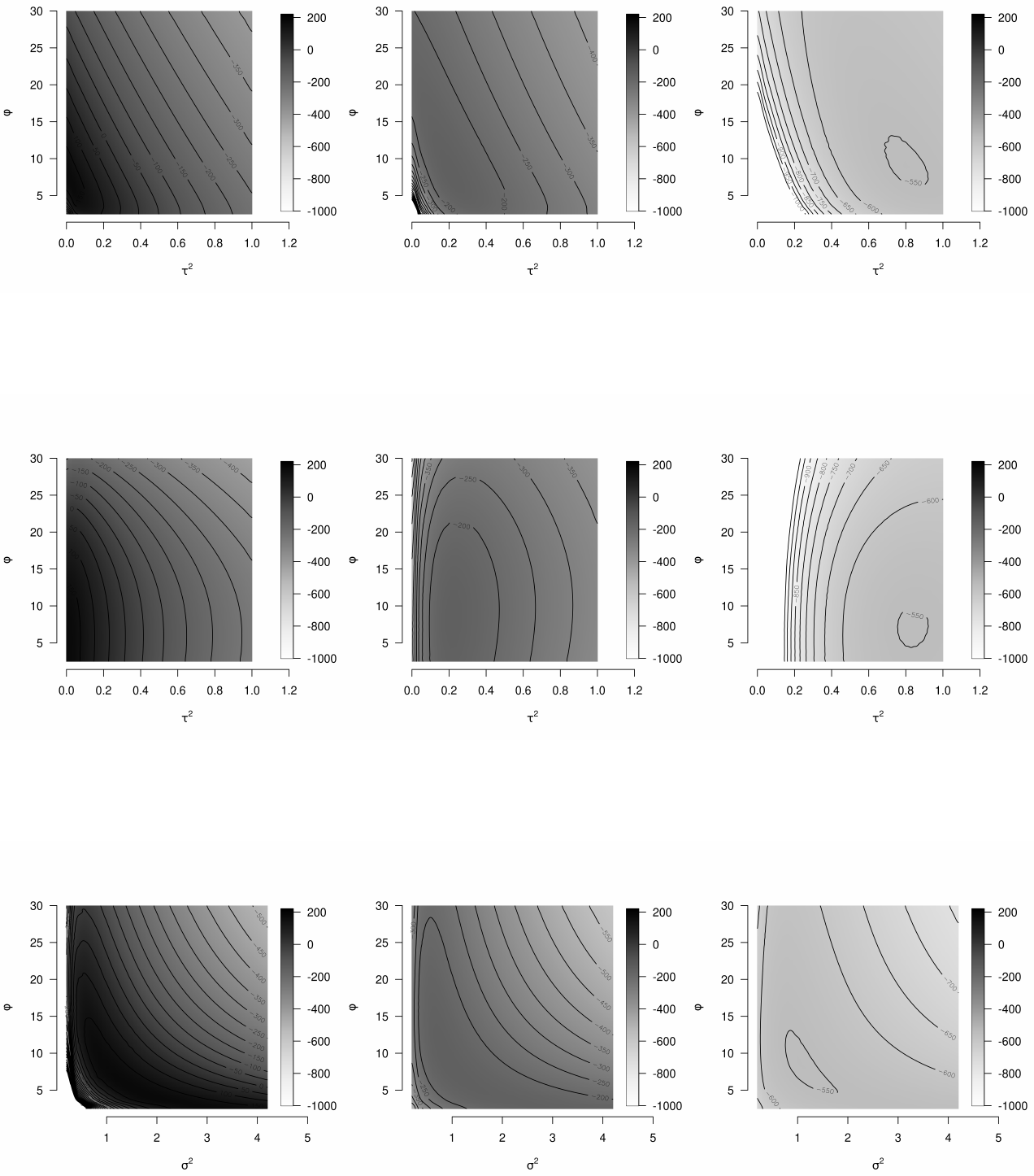}
	\caption{
{Interpolated maps of the log-likelihood. Darker shades indicate higher values. The first row corresponds to $\sigma^2 = \sigma^2_0 = 1$, the second row corresponds to $\sigma^2\phi = \phi_0 = 7.49$, and the third row corresponds to $\tau_0 = \tau^2_0$. The columns correspond to $\tau_0 = 0.0$, $\tau_0 = 0.2$, and $\tau_0 = 0.8$, respectively.} \label{fig:sim1}} 
\end{figure}

\subsection{Parameter estimation}\label{subsec: par_estimate} 
We use maximum likelihood estimators to illustrate the asymptotic properties of the parameter estimates. To find the maximum likelihood estimators of $\{\sigma^2, \tau^2, \phi, \kappa\}$, we use the log of the profile likelihood for $\phi$ and $\eta = \tau^2/\sigma^2$, given by
\begin{eqnarray}\label{eq: log-profile}
\log\{\mathcal{PL}(\phi, \eta)\}  \propto  &-& \frac{1}{2}\log[\det\{\rho(\phi) + \eta I_n\}] - \frac{n}{2} \nonumber \\
&-& 
\frac{n}{2} \log\left[\frac{1}{n}y^\top \{\rho(\phi) + \eta I_n\}^{-1}y\right] 
\end{eqnarray}
where $\log\{\mathcal{PL}(\phi, \eta)\} =\log[\underset{\sigma^2}{\sup} \{\mathcal{L}(\sigma^2, \phi, \eta)\}]$, $\rho(\phi)$ is the correlation matrix of the underlying process $w(\cdot)$ over observed locations $\chi_n$. We optimize \eqref{eq: log-profile} to obtain maximum likelihood estimators $\widehat{\phi}$ and $\widehat{\eta}$. The maximum likelihood estimator for $\sigma^2$ is $\widehat{\sigma}_n^2 = y^\top \{\rho(\widehat{\phi}) + \widehat{\eta} I_n\}^{-1}y/n$. Calculations were executed using the R function \texttt{optimx} using the Broyden-Fletcher-Goldfarb-Shanno algorithm \citep{fletcher2013practical} with $\phi > 0$ and $\eta > 0$, and $\eta=0$ for models without a nugget. 

\begin{figure}
	\vspace*{-5pt}
	\centering
	\includegraphics[width = \linewidth]{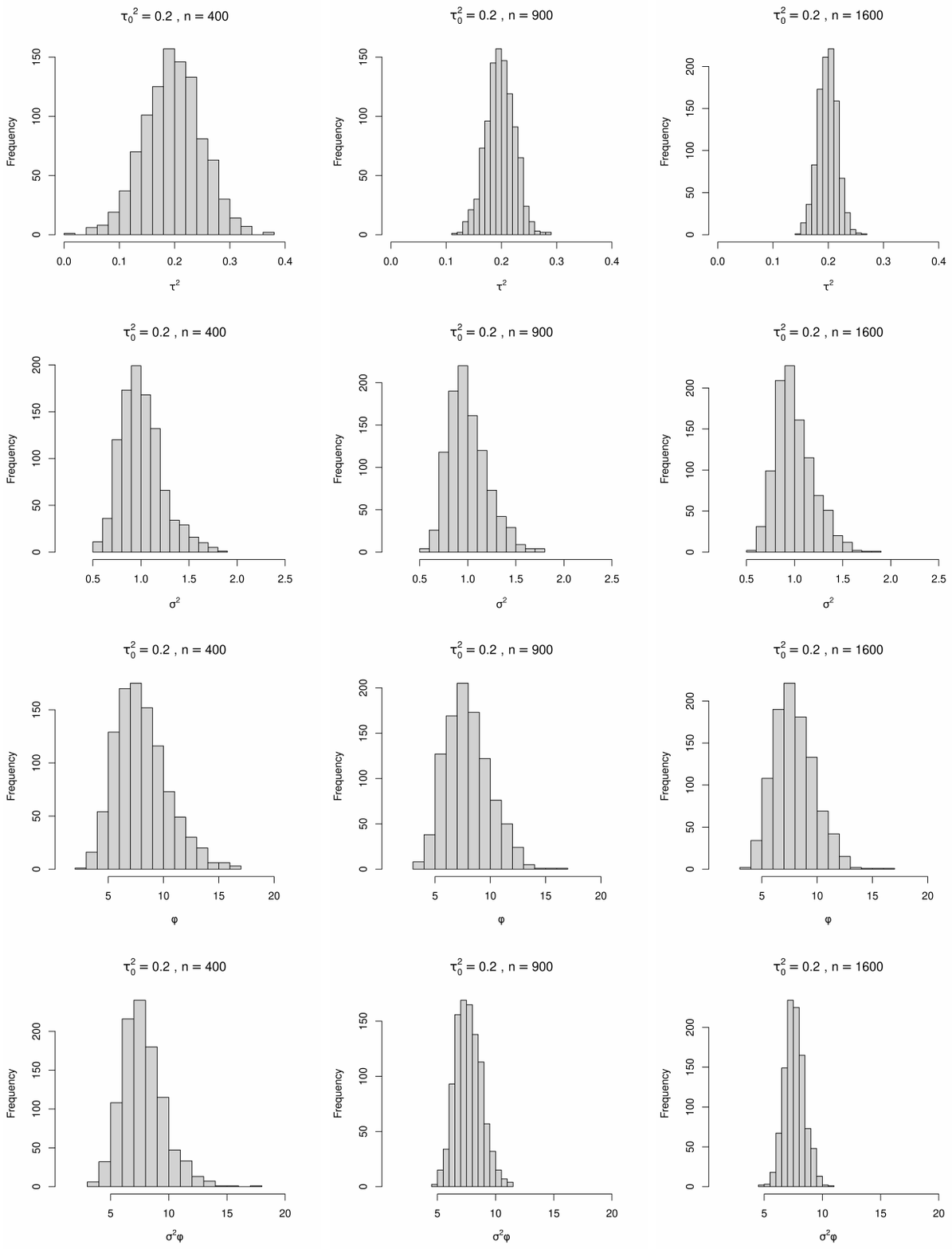}
	\caption{
{Histograms of $\tau^2$ (top row), $\sigma^2$ (second row), $\phi$ (third row) and $\kappa = \sigma^2\phi^{2\nu}$ (fourth row) obtained from simulation experiments with $\phi_0 = 7.49, \tau_0^2 = 0.2$.  \label{fig:sim2}}} 
\end{figure}

We calculate estimators for $\{\tau^2, \phi, \sigma^2, \kappa\}$ for each realization with sample sizes 400, 900 and 1600. For each parameter setting and sample size, there are 1000 estimators for $\{\tau^2, \phi, \sigma^2\}$ and $\kappa$. Figure~\ref{fig:sim2} depicts the histograms for the maximum likelihood estimators for $\tau^2$, $\phi$, $\sigma^2$ and $\kappa$ obtained from simulations with the parameter setting $\{\phi_0, \tau^2_0\} = \{7.49, 0.2\}$. There is an obvious shrinkage of the variance of estimators for $\tau^2$ and $\kappa$ as we increase the sample size from 400 to 1600. We also observe that their distribution becomes more symmetric with an increasing sample size. In contrast, the variance of the estimators for $\sigma^2$ and $\phi$ do not have a significant decrease as sample size increases. This is supported by the infill asymptotic results. The maximum likelihood estimators for $\tau^2$ and $\kappa$ are consistent {and asymptotically normal}. 
The maximum likelihood estimators for $\phi$ and $\sigma^2$ are not consistent and, hence, their variances do not decrease to zero with increasing sample size. 

Table~\ref{table:tau}--\ref{table:kappa} 
list percentiles, biases, and sample standard deviations for the estimates of $\tau^2$, $\phi$, $\sigma^2$ and $\kappa$ for each of the 9 parameter settings and offer further insights about the finite sample inference. When the spatial correlation is strong ($\phi$ is small), $\widehat{\tau}^2$ tends to be more precise, while $\widehat{\sigma}^2$ tends to have more variability. Unsurprisingly, the measurement error is easily distinguished from a less variable latent process $w(\cdot)$. Highly correlated realizations of $w(\cdot)$ results in less precise inference for $\sigma^2$. 
If the nugget is larger, then the estimators for $\phi$, $\sigma^2$ and $\kappa$ are less precise; the presence of measurement error weakens the precision of the estimates.

\subsection{Interpolation}\label{subsec: interpolation}
We use the kriging estimator in \eqref{eq:BLUE} and its mean squared prediction error (MSPE) in \eqref{eq:varpred} to explore spatial interpolation in the presence of the nugget. We use \eqref{eq:BLUE} to predict the underlying process $w(\cdot)$ over unobserved locations. From Theorem~8 in Chapter~3 of \cite{Stein99}, we expect a clear trend of convergence for $d = 1$. Let $\nu = 1/2$, $\tau_0^2 = 0.2$, $\sigma_0^2 = 1.0$ and $\phi_0 = 7.49$. We use \eqref{eq:basic} to generate observations over $12,000$ randomly picked locations in $[0, 1]$. We compute the MSPE using 3 hold-out points $\{0.25,0.5,0.75\}\in [0,1]$ for different subsets of the data with sample sizes ranging from $500$ to $12,000$. 
Figure~\ref{fig:sim3}(a) shows that the MSPE 
tends to approach $0$ as sample size increases. 
This corroborates Stein's conjecture that the underlying process $w(\cdot)$ in \eqref{eq:basic} can be consistently estimated on a finite interval. 

Next, we use the simulated data set with $n = 1600$ locations over the unit square used in Section~\ref{subsec: par_estimate}. We calculate the MSPE using \eqref{eq:varpred} and \eqref{eq:varpred2} over a $50 \times 50$ regular grid of locations over $[0,1]^2$. This is repeated for different data sets with sample sizes varying between 400 and 1600. Figure~\ref{fig:sim3}(b) shows that the MSPE decreases as sample size increases. This trend still holds when the predictor is formed under misspecified models, a finding similar to those in \cite{KS13} without the nugget. If $\nu$ is fixed at the true generating value, then predictions under any parameter setting are consistent and asymptotically efficient with no nugget effect. The proof in \cite{KS13} is based on \cite{Stein93}, hence their results do not carry over to our setting due to the discontinuity in our covariogram at $0$. (This technical difficulty was also pointed out by \cite[p.38]{Yakowitz85}). However, their results suggest empirical studies to explore the asymptotic properties of interpolation. 

To compare with results in \citet[Section 2.3]{KS13}, we examine two ratios
\begin{equation*}
\mbox{i)} \quad \frac{\var_{\tau_0^2, \sigma_0^2, \phi_0} \{ \widehat{z}_n(\tau_1^2, \sigma_1^2, \phi_1)- w_0\}}{\var_{\tau_0^2, \sigma_0^2, \phi_0} \{ \widehat{z}_n(\tau_0^2, \sigma_0^2, \phi_0)- w_0\}}, \quad  \mbox{and ii)} \quad \frac{\var_{\tau_1^2, \sigma_1^2, \phi_1} \{ \widehat{z}_n(\tau_0^2, \sigma_1^2, \phi_1)- w_0\}}{\var_{\tau_0^2, \sigma_0^2, \phi_0} \{ \widehat{z}_n(\tau_0^2, \sigma_1^2, \phi_1)- w_0\}}.
\end{equation*}
Figure~\ref{fig:sim3}(c) compares the ratio defined by i). This ratio tends to approach 1 only when $\tau_1^2 = \tau_0^2$ and $\kappa = \kappa_0$. Unlike the case with no nugget, asymptotic efficiency is only observed when the estimator is fitted under models with Gaussian measures equivalent to the generating Gaussian measure. Figure~\ref{fig:sim3}(d) plots the ratio defined by ii). 
As in Fig.~\ref{fig:sim3}(c), this ratio also tends to approach 1 only when $\tau_1^2 = \tau_0^2$, $\kappa = \kappa_0$. 
Based on our simulation study, we posit that the asymptotic efficiency and asymptotically correct estimation of MSPE hold only when $\tau_1^2 = \tau_0^2$, $\kappa = \kappa_0$
. 

\begin{figure}
	\vspace*{-20pt}
	\centering
	\subfloat[]{\includegraphics[width = 0.8\linewidth]{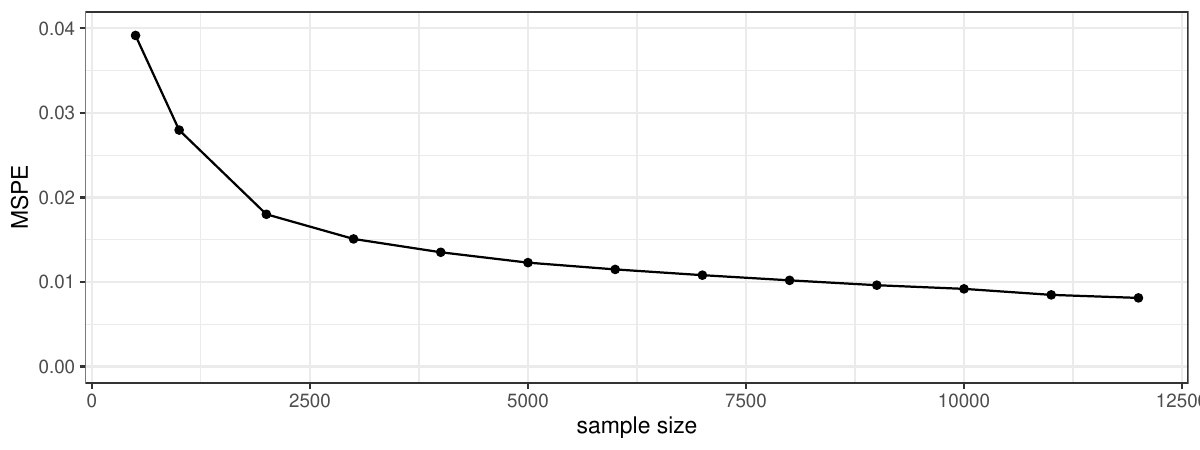}}\\
	\subfloat[]{\includegraphics[width = 0.8\linewidth]{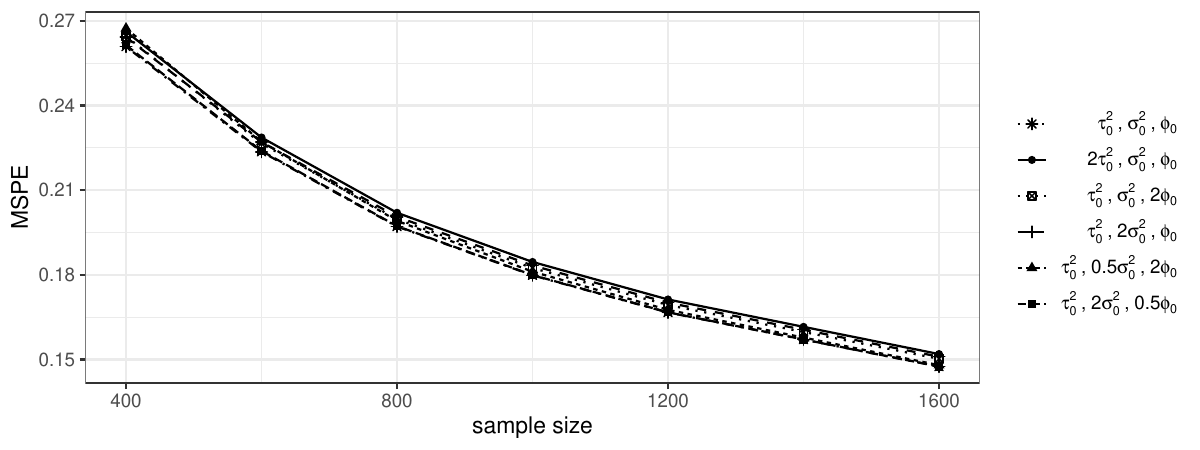}}\\
	\subfloat[]{\includegraphics[width = 0.8\linewidth]{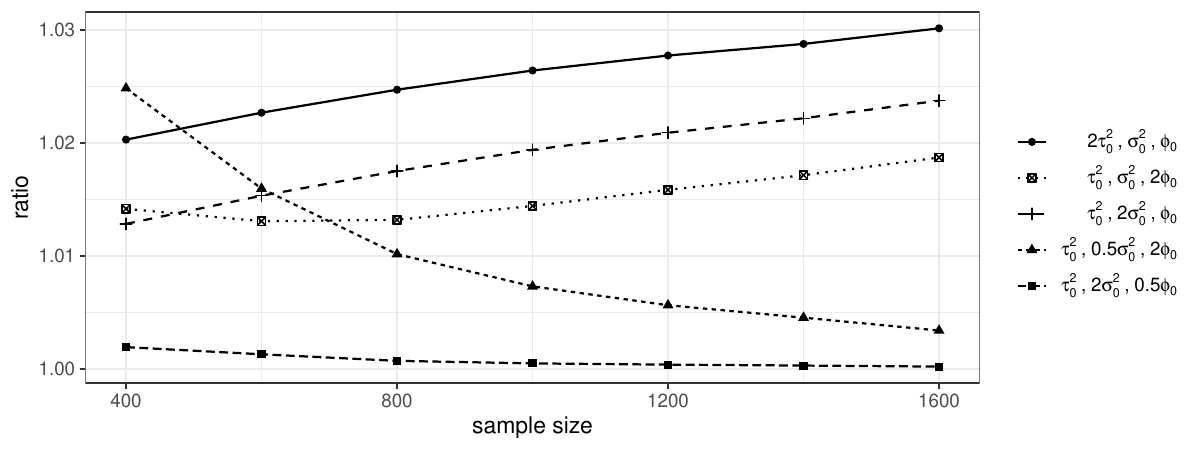}}\\
	\subfloat[]{\includegraphics[width = 0.8\linewidth]{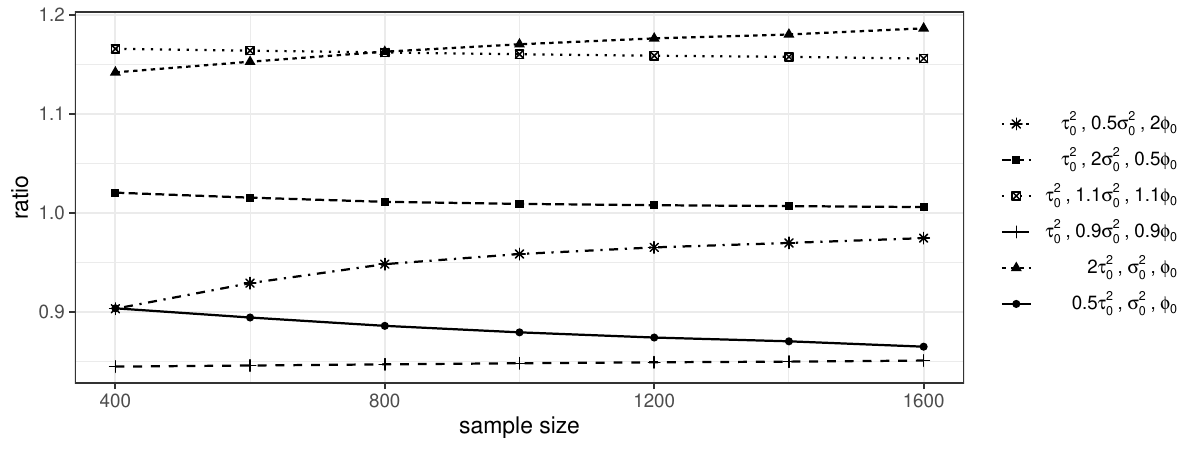}} 
	\caption{The MSPE for $w(\cdot)$ at (a) unobserved locations with study domain $[0, 1]$ (b) a $50 \times 50$ grid over $[0, 1]^2$. The ratio of mean square predict error (ratio) for testing asymptotic efficiency (c) and asymptotically correct estimation of MSPE (d) \label{fig:sim3}} 
\end{figure}

{
	\subsection{Bayesian inference from finite samples}\label{subsec: Bayes_sim}
	The asymptotic results in the preceding sections imply that a misspecified value of $\phi$ does not violate the consistency and asymptotic normality of the maximum likelihood estimator of the nugget $\tau^2$ or of the microergodic parameter $\kappa = \sigma^2\phi^{2\nu}$. In order to assess the extent to which these asymptotic results can guide practical implementation of model fitting for finite samples, we conduct a sensitivity test to check the stability of the inferences of $\tau^2$ and $\kappa$ from finite samples under different specifications for $\phi$. Here, we present inferences for $\tau^2$ and $\kappa$ based on a Bayesian analysis using finite samples.} 
	
	{
	We generate data over $n = 1,600$ observed locations situated on the perturbed grid described in Section~\ref{subsec: set-up}. We use a zero-centered Mat\'ern model with measurement error to generate the data, where $\nu = 1/2$, $\sigma^2 = 1$, $\tau^2 = 0.5$ and $\phi = 9.98$. We fit the simulated data through a zero-centered Mat\'ern model with measurement error with $\textrm{IG}(2, 1/2)$ and $\textrm{IG}(2, 1)$ priors for $\tau^2$ and $\sigma^2$, respectively. When assuming $\phi$ is unknown, we use a Gamma prior with shape 2 and rate $2/\phi_0$ for $\phi$, where $\phi_0$ is the true value of $\phi$ for the simulated data. We specified prior distributions with means equal to the data generating parameter values. We also fit the model with $\phi$ equal to 0.2, 0.5, 1, 2, and 5 times the value of $\phi_0$. We randomly select $n = 400$, $900$ and $1,600$ samples for model fitting. The posterior inferences are based on 4 MCMC chains, each with 500 iterations for burn-in and 500 iterations for sampling. All models are implemented in \texttt{cmdstanr} \citep{cmdstanr}. The reported $\hat{R}$ (R-hat) values for all parameters are no more than 1.02 and the reported effective sample size for all parameters are greater than 400, showing adequate convergence of all MCMC chains.}

{
	Figure~\ref{fig:sim_Bayes} illustrates the posterior distributions of $\tau^2$ and $\kappa$. As expected from Theorem~\ref{approx_dist}, the variance of the posterior distributions decrease with increasing values of $n$. The posterior distributions for $\kappa$ and $\tau^2$ approach the truth as $n$ increases, but the inference can be highly biased when $\phi$ is misspecified. The results for $\kappa$ are similar to those reported by \cite{KS13} for a zero-centered Mat\'ern model without measurement error. We observe stabler posterior inference of $\tau^2$ than $\kappa$ for the cases when $\phi$ is unknown or fixed at values no more than $\phi_0$. The case when $\phi = 5\phi_0$ calls for some additional remarks. Here, the effective spatial range (i.e., the distance beyond which the spatial correlation drops to $0.05$) is only about $4\%$ of the maximum inter-site distance in our domain. Hence, the spatial correlation is negligible making it difficult to distinguish the nugget $\tau^2$ from the ``partial sill'' $\sigma^2$ and inference is sensitive to the prior specification. This is a plausible explanation for the poorer estimates of $\tau^2$ when $\phi = 5\phi_0$.
}
{
	\begin{figure}[h]
		\vspace*{-20pt}
		\centering
		\subfloat[]{\includegraphics[width = 0.8\linewidth]{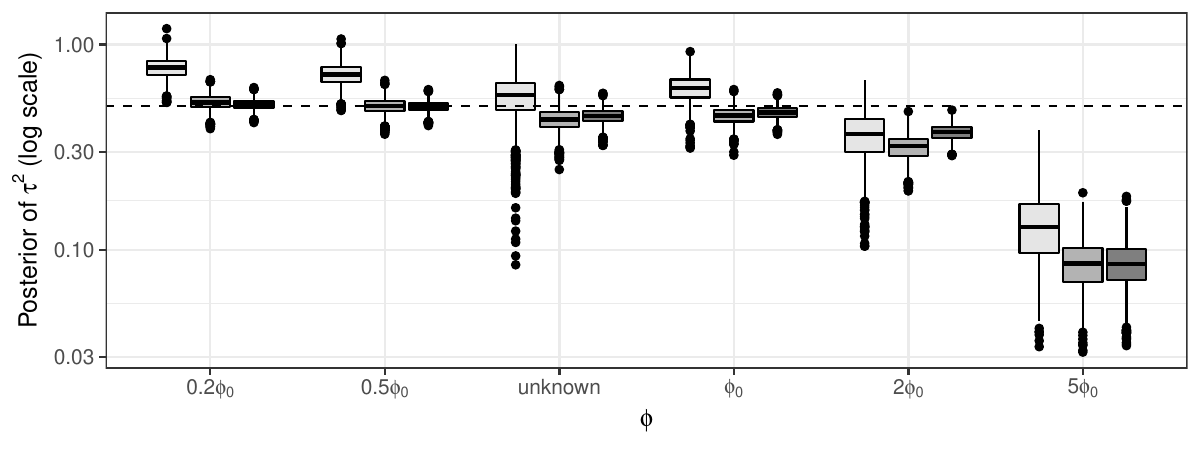}}\\
		\subfloat[]{\includegraphics[width = 0.8\linewidth]{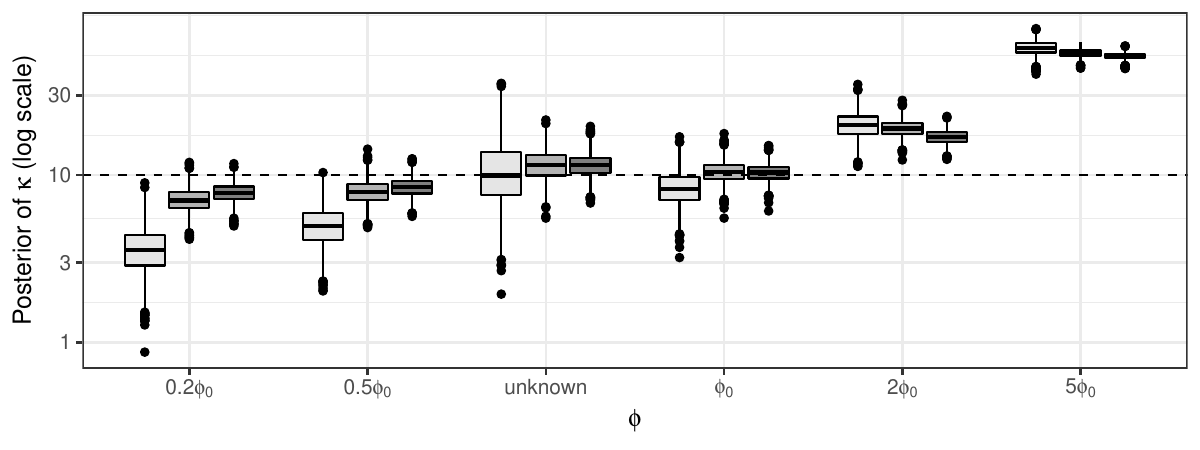}}
		\caption{{Posterior distributions for (a) $\tau^2$ and (b) $\kappa$ obtained from the simulation studies in Section~\ref{subsec: Bayes_sim}. The decay parameter is either estimated via MCMC sampling (unknown), fixed at the true value $\phi_0$, or fixed at a multiples of $\phi_0$, viz. $\{0.2\phi_0, \ldots, 5\phi_0\}$. The three boxplots in each group correspond to sample sizes of $n = 400$,  $900$, and $1,600$ reading from left to right. The dashed line indicates the true value.} \label{fig:sim_Bayes}} 
\end{figure}}

\section{Discussion}
We have developed insights into inference under infill asymptotics of Gaussian process parameters in the context of spatial or geostatistical analysis in the presence of the nugget effect. Our work can be regarded as an extension of similar investigations without the nugget effect. {While geostatistical modelling usually applies to $\mathbb{R}^d$ with $d \le 3$, we have also developed some new insights into $d\ge 5$, where consistency of the MLE's for the Mat\'ern model remains unresolved even without the nugget. 
}


We have discussed the complications in establishing consistency and asymptotic efficiency in parameter estimation and spatial prediction due to the discontinuity introduced by the nugget. Tools in standard spectral analysis no longer work in this scenario. Understanding the behaviour of such processes will enhance our understanding of identifiability of process parameters. For example, the failure to consistently estimate certain (non-microergodic) parameters can also be useful for Bayesian inference where we can conclude that the effect of the likelihood will never overwhelm the prior when calculating the posterior distribution of non-microergodic parameters. 
{Section~\ref{subsec: Bayes_sim} presented some insights into the behaviour of Bayesian estimates for the nugget in the presence of a misspecified range parameter. Formal investigations into the consistency of the posterior distributions of Mat\'ern covariogram parameters are certainly of interest and can be built upon some of our developments in the current manuscript.} 

We anticipate 
further research in variants of geostatistical models with the nugget. For example, 
one can 
explore whether some results, such as Theorem~2 in \cite{KS13} where $\phi$ is estimated, will hold for the Mat\'ern model with the nugget. Our simulations also suggest further research in asymptotic efficiency provided in Theorem~3 of \cite{KS13} in the presence of the nugget. With recent interest in scalable Gaussian process models, we can investigate asymptotic properties of approximations indicated on the lines of \cite{ve88} and Section~10.5.3 in \cite{zhang2012asymptotics}; \citep[also see][for scalable spatial process models in Bayesian settings]{Banerjee2017}. In Bayesian contexts, understanding posterior consistency for the nugget will offer insights into classes of priors. Finally, we point out that the conditions in Assumptions~\ref{assump:1}~and~\ref{assump:2} about eigenvalue estimates are expected and their rigorous proofs will constitute future research, {as will further theoretical explorations on Gaussian processes in $\mathbb{R}^d$ for all values of $d$.}. In particular, a rigorous proof of Assumption~\ref{assump:2} is challenging and will be of interest in general kernel methods and bandit problems. 



\section*{Acknowledgements}
We thank Robert Schaback for various pointers to the literature and stimulating discussions. We thank the Editor, the Associate Editor, and two anonymous referees for several useful suggestions that have helped improve the manuscript. The work of the authors were supported, in part, by federal grants NSF/DMS 1916349, 2113778 and 2113779;  NSF/IIS 1562303; and NIH/NIEHS 1R01ES027027.

\bibliographystyle{plainnat}
\bibliography{unique}  

\begin{thebibliography}{31}
\providecommand{\natexlab}[1]{#1}
\providecommand{\url}[1]{\texttt{#1}}
\expandafter\ifx\csname urlstyle\endcsname\relax
  \providecommand{\doi}[1]{doi: #1}\else
  \providecommand{\doi}{doi: \begingroup \urlstyle{rm}\Url}\fi

\bibitem[Abramowitz and Stegun(1965)]{AS65}
Milton Abramowitz and Irene~A Stegun.
\newblock \emph{Handbook of {M}athematical {F}unctions: with {F}ormulas,
  {G}raphs, and {M}athematical {T}ables}.
\newblock Dover, 1965.

\bibitem[Anderes(2010)]{A10}
Ethan Anderes.
\newblock On the consistent separation of scale and variance for {G}aussian
  random fields.
\newblock \emph{The Annals of Statistics}, 38\penalty0 (2):\penalty0 870--893,
  2010.

\bibitem[Banerjee(2017)]{Banerjee2017}
Sudipto Banerjee.
\newblock {High-Dimensional Bayesian Geostatistics}.
\newblock \emph{Bayesian Analysis}, 12\penalty0 (2):\penalty0 583 -- 614, 2017.
\newblock \doi{10.1214/17-BA1056R}.
\newblock URL \url{https://doi.org/10.1214/17-BA1056R}.

\bibitem[Belkin(2018)]{B18}
Mikhail Belkin.
\newblock Approximation beats concentration? {A}n approximation view on
  inference with smooth radial kernels.
\newblock In \emph{Conference On Learning Theory, {COLT} 2018}, pages
  1348--1361, 2018.

\bibitem[Bevilacqua et~al.(2019)Bevilacqua, Faouzi, Furrer, and Porcu]{BFFP19}
Moreno Bevilacqua, Tarik Faouzi, Reinhard Furrer, and Emilio Porcu.
\newblock Estimation and prediction using generalized {W}endland covariance
  functions under fixed domain asymptotics.
\newblock \emph{Ann. Statist.}, 47\penalty0 (2):\penalty0 828--856, 2019.

\bibitem[Chen et~al.(2000)Chen, Simpson, and Ying]{CSY00}
Huann-Sheng Chen, Douglas~G. Simpson, and Zhiliang Ying.
\newblock Infill asymptotics for a stochastic process model with measurement
  error.
\newblock \emph{Statistica Sinica}, pages 141--156, 2000.

\bibitem[Du et~al.(2009)Du, Zhang, and Mandrekar]{DZM09}
Juan Du, Hao Zhang, and VS~Mandrekar.
\newblock Fixed-domain asymptotic properties of tapered maximum likelihood
  estimators.
\newblock \emph{The Annals of Statistics}, 37\penalty0 (6A):\penalty0
  3330--3361, 2009.

\bibitem[Durrett(2019)]{Durrett}
Rick Durrett.
\newblock \emph{Probability---theory and examples}, volume~49 of
  \emph{Cambridge Series in Statistical and Probabilistic Mathematics}.
\newblock Cambridge University Press, 2019.
\newblock Fifth edition.

\bibitem[Etemadi(2006)]{E06}
Nasrollah Etemadi.
\newblock Convergence of weighted averages of random variables revisited.
\newblock \emph{Proceedings of the American Mathematical Society}, 134\penalty0
  (9):\penalty0 2739--2744, 2006.

\bibitem[Fletcher(2013)]{fletcher2013practical}
Roger Fletcher.
\newblock \emph{Practical methods of optimization}.
\newblock John Wiley \& Sons, 2013.

\bibitem[Furrer et~al.(2006)Furrer, Genton, and Nychka]{fur06}
Reinhard Furrer, Marc~G Genton, and Douglas Nychka.
\newblock Covariance tapering for interpolation of large spatial datasets.
\newblock \emph{Journal of Computational and Graphical Statistics},
  15:\penalty0 503--523, 2006.

\bibitem[Gabry and Češnovar(2020)]{cmdstanr}
Jonah Gabry and Rok Češnovar.
\newblock \emph{cmdstanr: R Interface to 'CmdStan'}, 2020.
\newblock https://mc-stan.org/cmdstanr, https://discourse.mc-stan.org.

\bibitem[Ibragimov and Rozanov(1978)]{IR78}
Ildar~Abdulovich Ibragimov and Yurii~Antol'evich Rozanov.
\newblock \emph{Gaussian random processes}, volume~9 of \emph{Applications of
  Mathematics}.
\newblock Springer-Verlag, New York-Berlin, 1978.
\newblock Translated from the Russian by A. B. Aries.

\bibitem[Jerome(1972)]{J72}
Joseph~W Jerome.
\newblock Asymptotic estimates of the n-widths in {H}ilbert space.
\newblock \emph{Proceedings of the American Mathematical Society}, 33\penalty0
  (2):\penalty0 367--372, 1972.

\bibitem[Kaufman et~al.(2008)Kaufman, Schervish, and Nychka]{KSN08}
Cari~G. Kaufman, Mark~J. Schervish, and Douglas~W. Nychka.
\newblock Covariance tapering for likelihood-based estimation in large spatial
  data sets.
\newblock \emph{J. Amer. Statist. Assoc.}, 103\penalty0 (484):\penalty0
  1545--1555, 2008.

\bibitem[Kaufman and Shaby(2013)]{KS13}
CG~Kaufman and Benjamin~Adam Shaby.
\newblock The role of the range parameter for estimation and prediction in
  geostatistics.
\newblock \emph{Biometrika}, 100\penalty0 (2):\penalty0 473--484, 2013.

\bibitem[Ma and Bhadra(2019)]{MB19}
Pulong Ma and Anindya Bhadra.
\newblock Kriging: {B}eyond {M}at\'ern.
\newblock \emph{arXiv preprint arXiv:1911.05865}, 2019.

\bibitem[Mat{\'e}rn(1986)]{Matern86}
Bertil Mat{\'e}rn.
\newblock \emph{Spatial {V}ariation}.
\newblock Springer-Verlag, 1986.

\bibitem[Santin and Schaback(2016)]{SS16}
Gabriele Santin and Robert Schaback.
\newblock Approximation of eigenfunctions in kernel-based spaces.
\newblock \emph{Advances in Computational Mathematics}, 42\penalty0
  (4):\penalty0 973--993, 2016.

\bibitem[Schaback(1995)]{Schaback95}
Robert Schaback.
\newblock Error estimates and condition numbers for radial basis function
  interpolation.
\newblock \emph{Advances in Computational Mathematics}, 3\penalty0
  (3):\penalty0 251--264, 1995.

\bibitem[Stein(1988)]{Stein88}
Michael~L Stein.
\newblock Asymptotically efficient prediction of a random field with a
  misspecified covariance function.
\newblock \emph{The Annals of Statistics}, pages 55--63, 1988.

\bibitem[Stein(1993)]{Stein93}
Michael~L Stein.
\newblock A simple condition for asymptotic optimality of linear predictions of
  random fields.
\newblock \emph{Statistics \& Probability Letters}, 17\penalty0 (5):\penalty0
  399--404, 1993.

\bibitem[Stein(1999)]{Stein99}
Michael~L Stein.
\newblock \emph{Interpolation of {S}patial {D}ata: {S}ome {T}heory for
  {K}riging}.
\newblock Springer-Verlag, 1999.

\bibitem[Vecchia(1988)]{ve88}
Aldo~V Vecchia.
\newblock Estimation and model identification for continuous spatial processes.
\newblock \emph{Journal of the Royal Statistical society, Series B},
  50:\penalty0 297--312, 1988.

\bibitem[Wang et~al.(2011)Wang, Loh, et~al.]{wang2011fixed}
Daqing Wang, Wei-Liem Loh, et~al.
\newblock On fixed-domain asymptotics and covariance tapering in gaussian
  random field models.
\newblock \emph{Electronic Journal of Statistics}, 5:\penalty0 238--269, 2011.

\bibitem[Wendland(1995)]{wendland95}
Holger Wendland.
\newblock Piecewise polynomial, positive definite and compactly supported
  radial functions of minimal degree.
\newblock \emph{Advances in computational Mathematics}, 4\penalty0
  (1):\penalty0 389--396, 1995.

\bibitem[Yakowitz and Szidarovszky(1985)]{Yakowitz85}
SJ~Yakowitz and F~Szidarovszky.
\newblock A comparison of kriging with nonparametric regression methods.
\newblock \emph{Journal of Multivariate Analysis}, 16\penalty0 (1):\penalty0
  21--53, 1985.

\bibitem[Ying(1991)]{Ying91}
Zhiliang Ying.
\newblock Asymptotic properties of a maximum likelihood estimator with data
  from a {G}aussian process.
\newblock \emph{Journal of Multivariate Analysis}, 36\penalty0 (2):\penalty0
  280--296, 1991.

\bibitem[Zhang(2004)]{Zhang04}
Hao Zhang.
\newblock Inconsistent estimation and asymptotically equal interpolations in
  model-based geostatistics.
\newblock \emph{Journal of the American Statistical Association}, 99\penalty0
  (465):\penalty0 250--261, 2004.

\bibitem[Zhang(2012)]{zhang2012asymptotics}
Hao Zhang.
\newblock Asymptotics and computation for spatial statistics.
\newblock In \emph{Advances and Challenges in Space-time Modelling of Natural
  Events}, pages 239--252. Springer, 2012.

\bibitem[Zhang and Zimmerman(2005)]{zhang2005towards}
Hao Zhang and Dale~L Zimmerman.
\newblock Towards reconciling two asymptotic frameworks in spatial statistics.
\newblock \emph{Biometrika}, 92\penalty0 (4):\penalty0 921--936, 2005.

\end{thebibliography}

\newpage
\begin{table}
	\caption{Summary of estimates of $\tau^2$: percentiles, bias, and sample standard deviations (SD).}\label{table:tau}
	\begin{tabular}{cccccccccc}
		$\tau^2_0$ & $\phi_0$ & n & 5\% & 25\% & 50\% & 75\% & 95\% & BIAS & SD \\\hline
		0.200 & 19.972 & 400 & 0.000 & 0.111 & 0.189 & 0.269 & 0.382 & -0.007  & 0.112 \\ 
		&  & 900 & 0.102 & 0.159 & 0.197 & 0.235 & 0.289 & -0.004 & 0.056  \\ 
		&  & 1600 & 0.141 & 0.175 & 0.199 & 0.221 & 0.252 & -0.002 & 0.035 \\ 
			\\
		&  7.489 & 400 &  0.110 & 0.162 & 0.197 & 0.232 & 0.281 & -0.003  & 0.053 \\ 
		&  & 900 & 0.157 & 0.181 & 0.198 & 0.216 & 0.238 & -0.002 & 0.025  \\ 
			&  & 1600 & 0.170 & 0.187 & 0.199 & 0.211 & 0.227 & -0.001 & 0.017 \\ 
			\\
			&  2.996 & 400 & 0.152 & 0.177 & 0.196 & 0.217 & 0.248 & -0.003 & 0.029 \\ 
			&  & 900 & 0.173 & 0.188 & 0.199 & 0.212 & 0.227 & 0.000 & 0.017  \\ 
			&  & 1600 & 0.182 & 0.191 & 0.200 & 0.208 & 0.219 & 0.000 & 0.012  \\ 
			\\
			0.800 & 19.972 & 400 & 0.321 & 0.619 & 0.777 & 0.903 & 1.090 & -0.047 & 0.229 \\ 
			&  & 900 & 0.615 & 0.725 & 0.792 & 0.861 & 0.974 & -0.009 & 0.110  \\ 
			&  & 1600 & 0.682 & 0.746 & 0.795 & 0.841 & 0.910 & -0.006 & 0.069 \\ 
			\\
			&  7.489 & 400 & 0.582 & 0.714 & 0.789 & 0.859 & 0.974 & -0.015 & 0.114  \\ 
			&  & 900 & 0.689 & 0.752 & 0.794 & 0.835 & 0.897 & -0.006 & 0.065 \\ 
			&  & 1600 & 0.725 & 0.768 & 0.799 & 0.826 & 0.869 & -0.003 & 0.044  \\ 
			\\
			& 2.996 & 400 & 0.662 & 0.738 & 0.789 & 0.845 & 0.931 & -0.007 & 0.081 \\ 
			&  & 900 & 0.720 & 0.766 & 0.797 & 0.828 & 0.871 & -0.004 & 0.047  \\ 
			&  & 1600 & 0.737 & 0.775 & 0.799 & 0.823 & 0.856 & -0.002 & 0.036 \\
\hline 
		\end{tabular}
	\end{table}

	\begin{table}
		\caption{Summary of estimates of $\phi$: percentiles, bias, and sample standard deviations(SD)}\label{table:phi}
			\begin{tabular}{cccccccccc}
				\hline
				$\tau^2_0$ & $\phi_0$ & n & 5\% & 25\% & 50\% & 75\% & 95\% & BIAS & SD \\
				\hline
				0.000 & 19.972 & 400 & 16.151 & 18.355 & 19.992 & 21.798 & 25.003 & 0.223 & 2.708 \\ 
				&  & 900 &  16.706 & 18.642 & 20.072 & 21.548 & 23.928 & 0.182 & 2.185  \\ 
				&  & 1600 & 17.077 & 18.800 & 20.041 & 21.403 & 23.557 & 0.144 & 1.968 \\ 
				\\
				&  7.489 & 400 & 5.237 & 6.680 & 7.643 & 8.830 & 10.792 & 0.324 & 1.672 \\ 
				&  & 900 & 5.430 & 6.722 & 7.659 & 8.655 & 10.382 & 0.280 & 1.511 \\ 
				&  & 1600 & 5.520 & 6.730 & 7.664 & 8.687 & 10.245 & 0.255 & 1.450\\ 
				\\
				&  2.996 & 400 & 1.584 & 2.489 & 3.297 & 4.315 & 5.859 & 0.479 & 1.339 \\ 
				&  & 900 & 1.605 & 2.468 & 3.316 & 4.298 & 5.792 & 0.463 & 1.299 \\ 
				&  & 1600 & 1.624 & 2.490 & 3.259 & 4.279 & 5.613 & 0.448 & 1.281\\ 
				\\
				0.200 & 19.972 & 400 & 13.626 & 17.185 & 20.058 & 23.260 & 28.138 & 0.358 & 4.427\\ 
				&  & 900 & 15.117 & 17.938 & 20.059 & 22.188 & 26.097 & 0.221 & 3.321  \\ 
				&  & 1600 & 15.749 & 18.328 & 19.972 & 21.728 & 25.02 & 0.158 & 2.779 \\ 
				\\
				&  7.489 & 400 & 4.596 & 6.271 & 7.757 & 9.377 & 12.430 & 0.535 & 2.364 \\ 
				&  & 900 & 5.081 & 6.521 & 7.820 & 9.179 & 11.572 & 0.480 & 1.998 \\ 
				&  & 1600 & 5.195 & 6.557 & 7.774 & 9.079 & 11.391 & 0.410 & 1.838\\ 
				\\
				&  2.996 & 400 & 1.436 & 2.291 & 3.244 & 4.415 & 6.725 & 0.563 & 1.707\\ 
				&  & 900 & 1.534 & 2.383 & 3.243 & 4.269 & 6.405 & 0.48 & 1.518 \\ 
				&  & 1600 & 1.570 & 2.420 & 3.217 & 4.208 & 6.130 & 0.453 & 1.424 \\ 
				\\
				0.800 & 19.972 & 400 & 11.804 & 16.533 & 20.359 & 24.806 & 33.859 & 1.315 & 6.932 \\ 
				&  & 900 &  14.650 & 17.405 & 20.077 & 23.065 & 27.831 & 0.490 & 4.175 \\ 
				&  & 1600 & 15.340 & 17.911 & 20.197 & 22.544 & 26.195 & 0.396 & 3.352\\ 
				\\
				&  7.489 & 400 & 3.878 & 6.029 & 7.754 & 9.866 & 14.034 & 0.670 & 3.038\\ 
				&  & 900 & 4.468 & 6.266 & 7.745 & 9.317 & 12.249 & 0.475 & 2.402\\ 
				&  & 1600 & 4.691 & 6.430 & 7.735 & 9.142 & 11.663 & 0.405 & 2.157\\ 
				\\
				& 2.996 & 400 & 1.259 & 2.281 & 3.279 & 4.723 & 7.385 & 0.681 & 1.975\\ 
				&  & 900 & 1.443 & 2.364 & 3.249 & 4.38 & 7.199 & 0.603 & 1.771\\ 
				&  & 1600 &1.479 & 2.382 & 3.216 & 4.263 & 6.591 & 0.509 & 1.602\\
				\hline
			\end{tabular}
		\end{table}

		\begin{table}
			\caption{Summary of estimates of $\sigma^2$: percentiles, bias, and sample standard deviations(SD)}\label{table:sigma}
				\begin{tabular}{cccccccccc}
					\hline
					$\tau^2_0$ & $\phi_0$ & n & 5\% & 25\% & 50\% & 75\% & 95\% & BIAS & SD \\
					\hline
					0.000 & 19.972 & 400 &  0.835 & 0.928 & 0.992 & 1.063 & 1.172 & -0.004 & 0.103 \\ 
					&  & 900 &  0.859 & 0.938 & 0.997 & 1.063 & 1.155 & 0.001 & 0.091   \\ 
					&  & 1600 & 0.865 & 0.942 & 0.998 & 1.057 & 1.151 & 0.002 & 0.087 \\ 
					\\
					&  7.489 & 400 &   0.721 & 0.860 & 0.976 & 1.109 & 1.374 & 0.000 & 0.198  \\ 
					&  & 900 &  0.724 & 0.872 & 0.980 & 1.104 & 1.344 & 0.001 & 0.192\\ 
					&  & 1600 &  0.733 & 0.871 & 0.978 & 1.111 & 1.356 & 0.002 & 0.189\\ 
					\\
					&  2.996 & 400 &  0.527 & 0.700 & 0.905 & 1.217 & 1.856 & 0.014 & 0.446 \\ 
					&  & 900 &  0.532 & 0.708 & 0.900 & 1.216 & 1.843 & 0.010 & 0.427  \\ 
					&  & 1600 &  0.537 & 0.705 & 0.914 & 1.204 & 1.845 & 0.011 & 0.423 \\ 
					\\
					0.200 & 19.972 & 400 &  0.735 & 0.890 & 1.012 & 1.127 & 1.280 &  0.009 & 0.167 \\ 
					&  & 900 &  0.830 & 0.928 & 1.001 & 1.085 & 1.203 & 0.008 & 0.114 \\ 
					&  & 1600 &  0.860 & 0.941 & 1.000 & 1.071 & 1.170 & 0.008 & 0.097\\ 
					\\
					&  7.489 & 400 &   0.706 & 0.848 & 0.978 & 1.129 & 1.435 & 0.006 & 0.22 \\ 
					&  & 900 &  0.732 & 0.855 & 0.972 & 1.128 & 1.373 & 0.002 & 0.203 \\ 
					&  & 1600 &  0.731 & 0.857 & 0.970 & 1.116 & 1.374 & 0.000 & 0.195\\ 
					\\
					&  2.996 & 400 &  0.527 & 0.700 & 0.905 & 1.217 & 1.856 & 0.014 & 0.446\\ 
					&  & 900 &  0.532 & 0.708 & 0.900 & 1.216 & 1.843 & 0.010 & 0.427  \\ 
					&  & 1600 & 0.537 & 0.705 & 0.914 & 1.204 & 1.845 & 0.011 & 0.423 \\ 
					\\
					0.800 & 400 & 19.972 & 0.653 & 0.874 & 1.025 & 1.208 & 1.531 & 0.050 & 0.265 \\ 
					&  & 900 & 0.761 & 0.911 & 1.014 & 1.110 & 1.257 & 0.011 & 0.149 \\ 
					&  & 1600 & 0.826 & 0.931 & 1.009 & 1.085 & 1.197 & 0.009 & 0.113 \\ 
					\\
					&  7.489 & 400 & 0.640 & 0.848 & 1.004 & 1.174 & 1.487 & 0.027 & 0.263\\ 
					&  & 900 & 0.701 & 0.862 & 0.990 & 1.146 & 1.421 & 0.016 & 0.225\\ 
					&  & 1600 & 0.710 & 0.860 & 0.985 & 1.129 & 1.413 & 0.012 & 0.215\\ 
					\\
					& 2.996 & 400 & 0.482 & 0.715 & 0.955 & 1.254 & 1.916 & 0.047 & 0.482\\ 
					&  & 900 & 0.517 & 0.720 & 0.950 & 1.240 & 1.874 & 0.044 & 0.462 \\ 
					&  & 1600 & 0.524 & 0.735 & 0.968 & 1.250 & 1.839 & 0.045 & 0.449\\
					\hline
				\end{tabular}
			\end{table}
			
			\begin{table}
				\caption{Summary of estimates of $\kappa$: percentiles, bias, and sample standard deviations(SD)}%
					\begin{tabular}{cccccccccc}
						\hline
						$\tau^2_0$ & $\phi_0$ & n & 5\% & 25\% & 50\% & 75\% & 95\% & BIAS & SD \\
						\hline
						0.000 & 19.972 & 400 & 17.200 & 18.596 & 19.752 & 21.117 & 23.197 & -0.045 & 1.881\\ 
						&  & 900 & 18.098 & 19.221 & 19.957 & 20.798 & 21.974 & 0.035 & 1.177 \\ 
						&  & 1600 & 18.764 & 19.457 & 19.973 & 20.531 & 21.399 & 0.039 & 0.805\\ 
						\\
						&  7.489 & 400 & 6.538 & 7.092 & 7.499 & 7.943 & 8.568 & 0.032 & 0.619 \\ 
						&  & 900 & 6.903 & 7.236 & 7.500 & 7.784 & 8.146 & 0.018 & 0.387\\ 
						&  & 1600 & 7.061 & 7.317 & 7.491 & 7.680 & 7.979 & 0.013 & 0.280\\ 
						\\
						&  2.996 & 400 & 2.666 & 2.869 & 3.004 & 3.158 & 3.369 & 0.018 & 0.213\\ 
						&  & 900 & 2.780 & 2.915 & 3.001 & 3.103 & 3.254 & 0.012 & 0.142 \\ 
						&  & 1600 & 2.841 & 2.935 & 3.000 & 3.077 & 3.191 & 0.011 & 0.106 \\ 
						\\
						0.200 & 19.972 & 400 & 11.760 & 16.227 & 20.111 & 24.691 & 31.242 & 0.677 & 6.052 \\ 
						&  & 900 & 14.827 & 17.806 & 19.879 & 22.566 & 26.735 & 0.313 & 3.693 \\ 
						&  & 1600 & 16.421 & 18.434 & 19.943 & 21.624 & 24.404 & 0.186 & 2.528 \\ 
						\\
						&  7.489 & 400 &  5.116 & 6.546 & 7.552 & 8.825 & 11.045 & 0.268 & 1.802 \\ 
						&  & 900 & 5.999 & 6.843 & 7.605 & 8.404 & 9.645 & 0.177 & 1.110\\ 
						&  & 1600 & 6.197 & 7.033 & 7.585 & 8.141 & 9.085 & 0.105 & 0.850\\ 
						\\
						&  2.996 &400 & 2.010 & 2.546 & 3.040 & 3.533 & 4.322 & 0.092 & 0.716\\ 
						&  & 900 & 2.282 & 2.706 & 3.028 & 3.343 & 3.900 & 0.055 & 0.493 \\ 
						&  & 1600 & 2.434 & 2.779 & 3.012 & 3.292 & 3.724 & 0.040 & 0.384\\ 
						\\
						0.800 & 19.972 &400 & 8.846 & 15.161 & 20.858 & 28.202 & 47.108 & 3.314 & 12.319\\ 
						&  & 900 & 12.700 & 16.839 & 20.077 & 24.320 & 31.399 & 0.830 & 5.715 \\ 
						&  & 1600 & 14.846 & 17.751 & 20.215 & 22.941 & 26.997 & 0.530 & 3.888\\ 
						\\
						&  7.489 & 400 & 4.080 & 5.980 & 7.677 & 9.679 & 13.537 & 0.591 & 2.929\\ 
						&  & 900 & 5.084 & 6.394 & 7.626 & 8.923 & 10.918 & 0.269 & 1.808\\ 
						&  & 1600 & 5.598 & 6.675 & 7.622 & 8.546 & 10.030 & 0.169 & 1.361\\ 
						\\
						& 2.996 & 400 & 1.708 & 2.444 & 3.093 & 3.849 & 5.432 & 0.259 & 1.175\\ 
						&  & 900 &  1.999 & 2.626 & 3.114 & 3.666 & 4.534 & 0.185 & 0.789 \\ 
						&  & 1600 & 2.210 & 2.712 & 3.086 & 3.478 & 4.210 & 0.129 & 0.618 \\
						\hline
					\end{tabular}
					\label{table:kappa}
				\end{table}

\end{document}